\documentclass[twoside, 11pt]{article}
\usepackage[left=1.5in,textwidth=30pc,textheight=54.5pc]{geometry}
\usepackage{amsfonts, amsmath, amssymb, amsthm, mathtools, relsize}
\usepackage{hyperref}
\usepackage{graphicx}
\usepackage[normalem]{ulem} 
\usepackage{fullpage}
\usepackage{cite}
\usepackage{enumerate}
\usepackage{graphicx}
\usepackage{xcolor}
\usepackage{tikz} 
\usepackage{authblk}
\usepackage{bbm}

\usepackage[linewidth=1pt]{mdframed}

\usepackage[textwidth=1.4in]{todonotes}
\reversemarginpar
\usepackage{caption}

\definecolor{nb}{rgb}{.6, .176, 1}
\definecolor{sienna}{rgb}{1, 0, 0}
\definecolor{darkgreen}{rgb}{0, .5, 0}

\numberwithin{equation}{section}
\allowdisplaybreaks

\newtheorem{theorem}{Theorem}[section]
\newtheorem{proposition}[theorem]{Proposition}

\newtheorem{lemma}[theorem]{Lemma}
\newtheorem{example}[theorem]{Example}

\newtheorem{remark}{Remark}

\newcommand{\GG}{\mbox{${\mathcal G}$}}
\newcommand{\FF}{\mbox{${\mathcal F}$}}

\newcommand{\EE}{\mbox{${\mathcal E}$}}

\newcommand{\I}{\mbox{${\mathbb I}$}}

\newcommand{\vv}{\mbox{\bf v}}

\DeclareMathOperator{\Tr}{Tr}

\usepackage{amsfonts}

 \newcommand{\R}{\mathbb{R}}

 \newcommand{\N}{\mathcal{N}}
\usepackage{cancel}

\usepackage{ulem}

 \newcommand{\A}{\textbf{A}}
 \newcommand{\D}{\textbf{D}}

\newcommand{\bigzero}{\mbox{\normalfont\Large\bfseries 0}}
\newcommand{\rvline}{\hspace*{-\arraycolsep}\vline\hspace*{-\arraycolsep}}

 \newcommand{\1}{\mathbf{1}}
  \newcommand{\0}{\mathbf{0}}

\title{Urns with Multiple Drawings and Graph-Based Interaction}

\author[1]{Yogesh Dahiya\thanks{ph20002@iisermohali.ac.in}}
\author[2]{Neeraja Sahasrabudhe\thanks{neeraja@iisermohali.ac.in}}

\affil[1]{Indian Institute of Science Education and Research Mohali (IISERM), India}
\affil[2]{Indian Institute of Science Education and Research Mohali (IISERM), India}
\date{} 

\begin{document}
\maketitle

\begin{abstract}
Consider a finite undirected graph and place an urn with balls of two colours at each vertex. At every discrete time step, for each urn, a fixed number of balls are drawn from that same urn with probability $p$, and from a randomly chosen neighbour of that urn with probability $1-p$. Based on what is drawn, the urns then reinforce themselves or their neighbours. For every ball of a given colour in the sample, in case of P\'olya-type reinforcement, a constant multiple of balls of that colour is added while in case of Friedman-type reinforcement, balls of the other colour are reinforced. These different choices for reinforcement give rise to multiple models. In this paper, we study the convergence of the fraction of balls of either colour across urns for all of these models. We show that in most cases the urns synchronize, that is, the fraction of balls of either colour in each urn converges to the same limit almost surely. A different kind of asymptotic behaviour is observed on bipartite graphs. We also prove similar results for the case of finite directed graphs. \color{black}
\end{abstract}

\


\


\section{Introduction}\label{sec1}

The P\'olya urn process is widely used in applications across mathematics, physics, economics and engineering. Since its introduction in 1923 by George P\'olya, and further generalization by Bernard Friedman, the study of urn processes has given rise to a large body of work. More recently, there has been work on urns with multiple or even infinite colours \cite{MR3780393, MR3654806} and urns with multiple drawings \cite{MR3666709, MR3915423, MR2203815, MR3126664}. In \cite{MR3780393}, stochastic approximation techniques are used to prove convergence results and central limit theorems (CLTs) for an urn process with multiple drawings containing balls of multiple colours. While the authors mainly consider a balanced urn and prove convergence results for the fraction of balls of each colour, some partial results are obtained for a two-colour unbalanced urn as well. In \cite{MR2203815,MR3126664}, the authors consider a two-colour urn model (say, with white and black colour balls) wherein at each time step, a constant number of balls is drawn from the urn, and the reinforcement is done depending upon the composition of the balls of white and black colour in the drawn sample. While in \cite{MR2203815}, for every white/black ball in the sample, a fixed constant number of white/black balls are added back to the urn, in \cite{MR3126664}, the opposite kind of reinforcement takes place, that is, for every white/black ball in the sample, a fixed constant number of black/white balls are added back to the urn. In \cite{MR3126664}, the first and second moments of the fraction of white balls are calculated, along with an asymptomatic analysis of the higher moments of the number of white balls in the urn centred around its mean, and in \cite{MR2203815} results are obtained pertaining to the distribution of the limiting fraction of white balls in the urn.

The study of urn models has also been extended in the direction of multiple urns with dependent reinforcement processes; however, the focus has been on urn processes where only one ball is drawn from each urn at every time step. In \cite{MR3217785, crimaldi2016fluctuation}, the authors consider a system where the underlying network for interaction is a finite complete graph. At each time step, for each urn, a ball is sampled from that urn with probability $p$ or from the super urn (formed by combining all the urns of the system) with probability $1-p$. The reinforcement is of P\'olya-type and the authors prove that the urns synchronize almost surely, that is, the fraction of balls of either colour converges almost surely to the same limit across all the urns. In \cite{crimaldi2016fluctuation}, CLTs are proved and it is shown that the rate of convergence and asymptotic variance depend on the parameter $p$. Building on this, in \cite{aletti2017interacting, sahasrabudhe2016synchronization}, the authors study  synchronization and fluctuations of a similar interacting Friedman urn system. In \cite{kaur2023interacting}, the authors study a more general model with graph-based nearest neighbour interaction and a general reinforcement scheme that includes P\'olya and Friedman-type reinforcements as special cases. They show that under certain conditions on the graph structure and the reinforcement type, the fraction of balls of either colour converges almost surely to a deterministic limit. Interacting urns on finite graphs have also been studied in  \cite{MR3269167, MR3346459, MR2079914}. There is also some recent work in the direction of generalizing the study of interacting urns to interacting reinforced random walks and/or interacting reinforced stochastic processes over networks (see \cite{crimaldi2019synchronization, 10.1214/17-AAP1296, 10.3150/18-BEJ1092}).

In this paper, we study a two-colour interacting urn process that combines the \lq multiple drawings' aspect with graph-based interaction. We consider urns, containing balls of two colours, placed at the vertices of a finite undirected graph. For each urn, a fixed number of balls are drawn from that same urn with probability $p$ (for a fixed $p \in [0, 1]$) and from a randomly chosen neighbour of that urn with probability $1-p$. During the reinforcement phase, balls can be reinforced in a variety of ways: an urn can reinforce itself, it can reinforce all its neighbours, or it can reinforce itself as well as its neighbours. Further, as in \cite{MR2203815,MR3126664}, we classify the two types of reinforcement schemes as P\'olya-type and Friedman-type. When the reinforcement scheme is  P\'olya-type, for every white/black ball in the sample, $C > 0$ number of white/black balls are added. In contrast, in the case of Friedman-type reinforcement scheme, white reinforces black and black reinforces white. By including interaction with neighbours during the sampling as well as the reinforcement step, we get a dependence on the square of the scaled adjacency matrix which makes this model different from those studied earlier. This can further be generalized by studying sampling and reinforcement that goes beyond nearest neighbour influence.


 
We study the effects of the graph structure, the interaction parameter $p$ and the reinforcement type on the asymptotic behaviour of the urn system. In this paper, we mainly focus on Friedman-type reinforcement and obtain convergence, synchronization and fluctuation results for the fraction of balls of either colour in the urns. However, we also obtain convergence results for the P\'olya-type reinforcement scheme under the added assumption of the regularity of the underlying graph and an assumption on the initial configuration of the urns. 
For Friedman-type reinforcement, universal synchronization of urns to a deterministic limit of $1/2$ is observed when either $p$ lies strictly between $0$ and $1$, or when the reinforcement scheme is such that each urn as well as its neighbours is reinforced. In some cases (when $p=0$ or $p=1$), only partial synchronization takes place if the underlying graph is bipartite. By partial synchronization, we mean that urns belonging to a common partition in the bipartite graph synchronize with each other. In this case, we show that for Friedman-type reinforcement, the fraction of balls of white colour in each urn belonging to a particular partition converges to the same limit almost surely. Further, if the limits are $Z_1^*$ and $Z_2^*$ on the two partitions, then $Z^*_1+Z^*_2=1$. The convergence result for bipartite graphs is obtained under certain constraints on the graph and initial configuration of the urns. We observe that the \lq multiple drawings' aspect does not have any effect on the convergence and synchronization results; however, the limiting variance in the fluctuation theorems depends on both the size of the drawn sample and the interaction parameter $p$. \color{black}
 
 The organization of the paper is as follows. In Section \ref{Sec:prelim} we introduce the urn process and various models. In section \ref{Sec:conv} we state and prove the convergence of the fraction of balls of either colour in each urn for various models using the theory of stochastic approximation. We show that for a large class of models, the limit is $1/2$ almost surely. In section~\ref{Sec:FTSR}, we study the special case of Friedman-type self-reinforcement for $p=0$ on bipartite graphs and under some assumptions prove that the fraction of balls of either colour in each urn converges to the same limiting fraction almost surely provided the urns are placed at vertices of the same partition of the graph. The behaviour of the limit, in terms of its distributional properties etc., is not known. In section~\ref{Sec:Fluc} we state and prove the fluctuation results, and show that appropriately scaled differences between the fraction of balls of either colour and the limit $1/2$ converge to a Gaussian distribution in most cases. Finally, in section~\ref{Sec:Dir} we discuss the behaviour of some of these models on directed networks. We briefly recall some of the stochastic approximation results used in the paper in the \hyperref[appn]{Appendix}. Throughout the paper, we write $f(t) \sim g(t)$ to mean $f(t) = \mathcal{O} (g(t))$. For a matrix $B(t) \in \mathbb{R}^{n \times m}$, by $B(t) \sim g(t)$ we mean that every entry of the matrix has the same order, that is, $B_{ij}(t) \sim g(t)$ for  $1 \leq i \leq n$ and $1 \leq j \leq m$. $\| \cdot \|$ refers to the standard Euclidean norm, and for $k>0$, ${\bf 0 }$ denotes a $k$-dimensional row vector of zeroes and $\0_{k \times k}$ denotes a zero matrix with $k$ rows and $k$ columns. \color{black}

\section{Urn Process and Reinforcement Schemes} \label{Sec:prelim}
We consider a connected undirected graph $\mathcal{G}=([N], \EE)$ on $N$ vertices, where $[N]$ denotes the set $\{1, 2, \ldots, N \}$ and $\EE \subseteq [N] \times [N]$ denotes the set of edges. At each vertex, we place an urn with a nonzero number of  white and black balls. The configuration of each urn (the number of balls of white and black colour) is updated simultaneously at every discrete time-step. 

For a vertex $v$, we define its neighbourhood as $\N(v) = \{ w \in [N] : (w, v) \in \EE \}$. Define $d_v \coloneqq | \N(v)|$ as the degree of the vertex $v$. We assume that the graph is connected; therefore, $d_v > 0 \; \forall v \in [N]$.  Let $\A$ denote the adjacency matrix of the graph $\GG$, and let $\D$ denote the diagonal matrix consisting of the degrees of the vertices on the diagonal, that is, $\D_{ij} = d_i$ for $i = j$ and $0$ otherwise.  For a matrix $M$, denote by $\Tr(M)$ the trace of $M$. Denote by $\1$ an $N-$dimensional row vector with each entry being $1$, and for an event $A$, let  $\I_{A}$ denote the indicator function of $A$. Throughout this paper, whenever the graph $\mathcal{G}$ is bipartite, the two disjoint partitions of the graph are denoted by $V$ and $W$, with $V=\{1, \dots, m\}$ and $W=\{m+1, \dots, N\}$.

The urn process is a two-phase process, consisting of sampling and reinforcement. At each step, each urn samples from itself or from one of its randomly chosen neighbours. The reinforcement process can be characterized by two types of dependence: (1) neighbourhood dependence: reinforcement in each urn depends on other urns in its neighbourhood, and (2) replacement scheme: reinforcement depends on which colour reinforces which colour. We study three types of neighbourhood dependence.

\begin{enumerate}[(i)]
\item Self reinforcement: each urn reinforces itself.
\item Neighbour reinforcement: each urn reinforces its neighbours in the graph.
\item Self and neighbour reinforcement: each urn reinforces itself as well as its neighbours.
\end{enumerate}
Further, two types of reinforcement schemes are studied.
\begin{enumerate}
\item P\'olya-type reinforcement scheme: each colour reinforces itself. That is, for every ball of colour white in the sample, $C$ balls of the same colour are reinforced.
\item Friedman-type reinforcement scheme: each colour reinforces the opposite colour, and not itself. That is, for every ball of colour white in the sample, $C$ balls of black colour are reinforced.
\end{enumerate}

\

Note that due to the multiple drawing aspect of sampling, the replacement matrix of the P\'olya-type reinforcement described above is not a diagonal matrix as in the classical P\'olya urn model. The above terminology is consistent with \cite{MR2203815,MR3126664} where multiple drawing urn models have been studied.

We now describe the process by which the configuration of urns evolves. Fix $p \in [0, 1]$. At each discrete time step $t$, the following process takes place.
\begin{itemize}
\item {\bf Sampling} For each urn, $s$ number of balls are drawn simultaneously and  uniformly at random (with or without replacement) from that same urn with probability $p$ and from a randomly chosen urn from the neighbourhood  with probability $1-p$. The sampling event for each urn is assumed to be independent of sampling from every other urn. As we shall see later, the asymptotic behaviour of the urns (convergence of the fraction of balls of a particular colour) is the same regardless of whether sampling is done with replacement or without replacement of balls.

\


 \color{black}
\item {\bf Reinforcement} The $N$ samples, each consisting of $s$ balls, are collated.  Each sample corresponds to the urn from which the sample was collected from.  We study the following types of reinforcement processes.
\begin{enumerate}
\item {\bf P\'olya-type reinforcement} 
\begin{enumerate}
\item P\'olya-type self reinforcement (PTSR) model: For every white/black  ball in the drawn sample, $C$ white/black balls are added to that same urn from which the sample was collected.\color{black}
\item P\'olya-type neighbour reinforcement (PTNR) model: For every white/black ball in the drawn sample, $C$ white/black balls are added to all the neighbours of the urn from which the sample was collected. 
\item P\'olya-type self and neighbour reinforcement (PTSNR) model: For every white/black ball in the drawn sample, $C$ white/black balls are added to that urn as well as to all the neighbours of that urn.
\end{enumerate}
\item {\bf Friedman-type reinforcement.} 
\begin{enumerate}
\item Friedman-type self reinforcement (FTSR) model: For every white/black ball in the drawn sample, $C$ black/white  balls are added to that same urn to which the sample corresponds.
\item Friedman-type neighbour reinforcement (FTNR) model: For every white/black ball in the drawn sample, $C$ black/white  balls are added to all the neighbours of that urn.
\item Friedman-type self and neighbour reinforcement (FTSNR) model: For every white/black ball in the drawn sample, $C$ black/white balls are added to that urn as well as to all the neighbours of that urn.
\end{enumerate}

\end{enumerate}
\end{itemize}
Observe that in the case of FTSR with $p=0$, for each urn the balls are sampled from a randomly chosen neighbour and the urn is reinforced. In the case of FTNR with $p=1$, balls are sampled from each urn and the neighbours are reinforced. For large time $t$, these two are essentially the same processes, and therefore, they exhibit the same asymptotic behaviour. The same holds for PTSR with $p=0$ and PTNR with $p=1$.

The reinforcement in each urn is balanced. That is, a constant number of balls is added to each urn at every time step. Let $T_t (i)$ denote the total number of balls in the $i^{th}$ urn at time $t$. Observe that $T_t (i)$ is deterministic, and we have $T_{t+1}(i) = T_t(i) + Cs \omega_i$, where for $1 \leq i \leq N$,
\begin{equation*} \omega_i
 =  \begin{cases}   1& \  \text{in case of self reinforcement (PTSR and FTSR models)}, \\
                    d_i & \ \text{in case of neighbour reinforcement (PTNR and FTNR models)}, \\
 d_i +1 & \  \text{in case of self and neighbour reinforcement (PTSNR and FTSNR models)}.
\end{cases}
\end{equation*}
 While the linear increment of the total number of balls in each urn is essential for our argument, the results hold for cases when $T_t (i)$ is random, but such that $\frac{T_t (i)}{t}$ converges to a fixed constant. This can happen, for instance, when instead of reinforcing all the neighbours, a random neighbour is chosen and reinforced. This case is discussed briefly in Remark~\ref{rem:random}. \color{black}

Let $\Omega$ denote the diagonal matrix consisting of $\omega_i$'s on the diagonal, and let $W_t(i)$ and $Z_t(i)$ denote the number and fraction of white balls in the $i^{th}$ urn at time $t$, respectively. Define $Z_t:= (Z_t(1), \ldots, Z_t(N)) \in \mathbb{R}^N$.

\

Let $Y_{t+1}(i)$ be the number of white balls sampled from the $i^{th}$ urn  at time $t$. Then, the evolution of the number of white balls in the $i^{th}$ urn (for $1 \leq i \leq N)$ can be described as follows.
\begin{enumerate}
\item P\'olya-type: $W_{t+1}(i) = W_t(i) + C \left(  \eta Y_{t+1}(i)+\kappa \sum\limits_{j \in N(i)}  Y_{t+1}(j)\right) $,
\item Friedman-type: $W_{t+1}(i) = W_t(i) + C \left( \eta(s- Y_{t+1}(i)) + \kappa \sum\limits_{j \in N(i)} (s- Y_{t+1}(j)) \right)$.
\end{enumerate}
where $\kappa$ and $\eta$ depend on the kind of urn reinforcement in question. More precisely, we have
\begin{equation*} \kappa
 =  \begin{cases}   0 & \  \text{in case of self reinforcement (PTSR and FTSR models),} \\
                    1& \ \text{in case of neighbour reinforcement (PTNR and FTNR models),}\\
 1 & \  \text{in case of self and neighbour reinforcement (PTSNR and FTSNR models).}\\                   
\end{cases}
\end{equation*}

\begin{equation*} \eta
 =  \begin{cases}   1& \ \text{in case of self reinforcement (PTSR and FTSR models),} \\
                    0 & \ \text{in case of neighbour reinforcement (PTNR and FTNR models),}\\
 1 & \  \text{in case of self and neighbour reinforcement (PTSNR and FTSNR models).}\\                   
 
\end{cases}
\end{equation*}
Define a filtration $\{\FF_t\}_{t\geq 0}$ by $\FF_t = \sigma( Y_s(i) : i \in [N] \text{ and } s \leq t )$. For $1 \leq j \leq N$, and for $ 0 \leq k \leq s $,
\begin{equation} \label{Y}
 P( Y_{t+1}(j) = k| \FF_t)={s \choose k} \left\{ p Z_t(j)^k (1-Z_t(j))^{s-k} + \frac{1-p}{d_j} \sum\limits_{l \in N(j)}  Z_t(l)^k (1-Z_t(l))^{s-k}\right\},
 \end{equation}
when the sampling is done with replacement. In case of sampling without replacement, we have
\begin{equation} \label{Y_without replacement}
 P( Y_{t+1}(j) = k| \FF_t)=  p \frac{{W_t(j)\choose k} {{T_t(j)-W_t(j)}\choose {s-k}}}{{T_t(j)\choose s}} + \frac{1-p}{d_j} \sum\limits_{l \in N(j)}  \frac{{W_t(l)\choose k} {{T_t(j)-W_t(j)}\choose {s-k}}}{{T_t(j)\choose s}}.
 \end{equation}
Note that for any $j \in [N]$, the sequence $\{Y_t(j)\}_{t\geq 0}$ is not exchangeable. For $j \in [N]$, define 
$$\chi_t(j) =  \begin{cases} 
 Y_{t}(j) & \text{in case of P\'olya-type reinforcement}, \\
 s-Y_t(j) & \text{in case of Friedman-type reinforcement}.
\end{cases}$$

\

\noindent For any $i \in [N]$, we have
\begin{eqnarray} \label{Zi}
Z_{t+1}(i) &=& \frac{W_{t+1}(i)}{T_{t+1}(i)} \nonumber \\
&=& \frac{1}{T_{t+1}(i)} \left( Z_t(i)T_t(i) + C \left( \eta \chi_{t+1}(i)+ \kappa \sum\limits_{j \in N(i)} \chi_{t+1}(j) \right)  \right).
\end{eqnarray}
This implies
\begin{eqnarray}\label{ExpZi}
E[Z_{t+1}(i) | \FF_t] &=& \frac{1}{T_{t+1}(i)} \left[ Z_t(i) T_t(i) + C\left(\eta E[\chi_{t+1}(i) | \FF_t]+ \kappa \sum\limits_{j \in N(i)} E[\chi_{t+1}(j) | \FF_t]\right)  \right]. \nonumber 
\end{eqnarray}
Using \eqref{Y} and \eqref{Y_without replacement}, we get that for $i \in [N]$,
\begin{equation} \label{Expchi}
E[\chi_{t+1}(i) | \FF_t] = s \times \begin{cases}  p Z_t(i) + \frac{1-p}{d_i} \sum\limits_{l \in N(i)} Z_t(l) & \text{for P\'olya-type reinforcement,}  \\
1 - p Z_t(i) - \frac{1-p}{d_i} \sum\limits_{l \in N(i)} Z_t(l) & \text{for Friedman-type reinforcement}. 
\end{cases}
\end{equation}
The asymptotic behaviour of any urn $i \in [N]$ depends on $E[\chi_{t+1}(i) | \FF_t]$, which is the expected reinforcement in that urn at that time given the past. Since \eqref{Expchi} turns out to be the same regardless of whether sampling is done with replacement of balls or without replacement, all our results hold for both kinds of sampling unless stated otherwise. \color{black} We now write the recursion for $Z_t$ as an $N$-dimensional stochastic approximation scheme (see \hyperref[appn]{Appendix}). Using \eqref{Zi}, for $i \in [N]$ we have
\begin{eqnarray} \label{recur}
 Z_{t+1}(i) &=& \frac{T_t(i)}{T_{t+1}(i)} Z_t(i) + \frac{C}{T_{t+1}(i)} \left( \eta \chi_{t+1} (i)  +  \kappa \sum\limits_{j \in N(i)} \chi_{t+1} (j) \right)   \\
 &=& Z_t(i) + \frac{C}{T_{t+1}(i)}\left( \eta \chi_{t+1} (i) - s \omega_i Z_t(i)  +  \kappa \sum\limits_{j \in N(i)} \chi_{t+1} (j) \right)  \nonumber \\ 
  &=& Z_t(i) + \frac{1}{t+1} \left[ \frac{1}{s \omega_i} \left( \eta \chi_{t+1} (i)  +  \kappa \sum\limits_{j \in N(i)} \chi_{t+1} (j) \right) - Z_t(i)\right] + \epsilon_t(i),\nonumber
  \end{eqnarray}
  where $\epsilon_t(i) = \left(\frac{Cs\omega_i}{T_{t+1}(i)} - \frac{1}{t+1}\right)\left[ \frac{1}{s \omega_i} \left( \eta \chi_{t+1} (i)  +  \kappa \sum\limits_{j \in N(i)} \chi_{t+1} (j) \right) - Z_t(i)\right]  \to 0$ as $t \to \infty$ (since \\
$T_{t+1}(i)= T_0(i) + Cs \omega_i (t+1)$).  
 For $j \in [N]$, define by $\Delta M_{t+1}(j) = \chi_{t+1}(j) - E[\chi_{t+1}(j) | \FF_t]$ the martingale difference. Then,
\begin{eqnarray*}
 Z_{t+1} (i) &=& Z_t(i) + \frac{1}{s\omega_i(t+1)}  \left( \eta E[\chi_{t+1}(i) | \FF_t] - s \omega_i Z_t(i) +  \kappa \sum\limits_{j \in N(i)} E[\chi_{t+1}(j) | \FF_t] \right)  \\
 &+&  \frac{1}{s \omega_i(t+1)} \left( \eta \Delta M_{t+1}(i) +\kappa \sum\limits_{j \in N(i)} \Delta M_{t+1}(j)   \right) + \epsilon_t(i).
 \end{eqnarray*}

\

We denote the row vectors $(\chi_t(1), \dots, \chi_t(N)), (\Delta M_t(1), \dots, \Delta M_t(N))$ and $(\epsilon_t(1), \dots, \epsilon_t(N))$ by $\chi_t$, $ \Delta M_t$ and $\epsilon_t$, respectively. Then, we can write the $N$-dimensional recursion for $Z_t$ as
 \begin{eqnarray}\label{recur_vector}
 Z_{t+1} &= & Z_{t} + \frac{1}{s(t+1)}  \left(  E[\chi_{t+1}| \FF_t]  \left(\eta I + \kappa \A\right){\Omega}^{-1}- s Z_t + \Delta M_{t+1} (\eta I + \kappa \A){\Omega}^{-1}\right) + \epsilon_t \nonumber \\
 & = & Z_{t} + \frac{1}{t+1}\left( h(Z_{t})+ \frac{\Delta M_{t+1}}{s} (\eta I + \kappa \A){\Omega}^{-1} \right) + \epsilon_t,  
 \end{eqnarray}
where $h(Z_t)= \frac{1}{s} \left(E[\chi_{t+1}| \FF_t]  \left(\eta I + \kappa \A \right){\Omega}^{-1} \right)-  Z_t $ is a Lipschitz function of $Z_t$ and $\epsilon_t \to 0$ as $t \to \infty$. Note that the matrix $\Omega$ equals $I$ in case of self reinforcement, $\D$ for  neighbour reinforcement, and $\D+I$ for the case of self and neighbour reinforcement. Using the appropriate value of $\kappa$ and $\eta$ based on the kind of reinforcement, the fact that $\1 \A \D^{-1}=\1$ (since $\A \D^{-1}$ is column stochastic), and using \eqref{Expchi}, we get 
\begin{equation} \label{h functions}
h(Z_t)  =  \begin{cases} (1-p)Z_t (\A \D^{-1}-I)   & \text{for the PTSR model,}  \\
 Z_t(p\A \D^{-1} +(1-p)(\A \D^{-1})^2-I)& \text{for the PTNR model,} \\
 Z_t\left[(pI + (1-p)\A \D^{-1})(\A+I)(I+\D)^{-1} -I \right] & \text{for the PTSNR model,}  \\
 \1 -Z_t((1+p)I + (1-p)\A \D^{-1}) & \text{for the FTSR model,}  \\
  \1 -Z_t(I + p\A \D^{-1}+ (1-p)(\A \D^{-1})^2 ) & \text{for the FTNR model,} \\
 \1- Z_t\left[(pI + (1-p)\A \D^{-1})(\A+I)(I+\D)^{-1} +I \right]  & \text{for the FTSNR model.}  
\end{cases}
 \end{equation}
 As mentioned before, the cases FTSR with $p=0$ and FTNR with $p=1$ (similarly, PTSR with $p=0$ and PTNR with $p=1$) do not require a separate analysis of the asymptotic behaviour. This is further clear from the fact that the $h(\cdot)$ function for these cases is identical. 
 
The stochastic approximation theory (see \hyperref[appn]{Appendix}) can be used to analyse the limiting behaviour of $Z_t$. In order to find the points of convergence for $Z_t$, we need to find the stable limit points of the O.D.E. $\dot{z}_t = h(z_t)$, which are a subset of  $\{z: h(z) = \0 \}$ . To determine the set of zeroes of $h$, we need to study the spectral properties of the matrices involved in $h(\cdot)$ for various urn models described above. In the next section, using the spectral analysis of the matrices involved and using the stochastic approximation theory, we study the asymptotic behaviour of $Z_t$ and characterize the limits for both P\'olya-type and Friedman-type reinforcement schemes.


 \section{ Convergence Results} \label{Sec:conv}
We begin by showing that for the Friedman-type reinforcement scheme, whenever there is a unique limit point of the O.D.E. $\dot{z}_t = h(z_t)$, it is stable. In other words, we show that the corresponding Jacobian matrices (denoted by $\frac{\partial h}{\partial z}$) are such that their eigenvalues have non-positive real parts. 

\

 \begin{lemma}\label{stable}
For Friedman-type reinforcement, the real part of each eigenvalue of $\frac{\partial h}{\partial z}$ is non-positive, where $h(\cdot)$ is as in \eqref{h functions}.
 \end{lemma}
 
\begin{proof} Since $h$ is a linear function of $z_t$, the Jacobian is independent of $z_t$ and we have
\begin{equation*}\label{Jacobian}
\frac{\partial h}{\partial z}  =  \begin{cases}   - I - (p I + (1-p)\A \D^{-1}) & \text{for the FTSR model,}  \\
-I -(p \A \D^{-1} + (1-p)(\A \D^{-1})^2 ) & \text{for the FTNR model,}\\
-I -(pI + (1-p)\A \D^{-1})(\A+I)(I+\D)^{-1}  & \text{for the FTSNR model}.
\end{cases}
\end{equation*}
Note that the matrices $\A\D^{-1}$, $(\A+I) (I+\D)^{-1}$ and  $p I + (1-p)\A \D^{-1}$ are column stochastic, and since the product of two column stochastic matrices is column stochastic, the matrices $p \A \D^{-1} + (1-p)(\A \D^{-1})^2$ and $(pI + (1-p)\A \D^{-1})(\A+I)(I+\D)^{-1}$ are also column stochastic. So, it is enough to show that the lemma is satisfied for matrices of the form $-I-X$ where $X$ is column stochastic.
 Let $\lambda$ be an arbitrary eigenvalue of $-I-X$. Then $-\lambda -1$ is an eigenvalue of the matrix  $X$, which is column stochastic. Therefore, $| \lambda + 1 | \leq 1$. This implies $ (Re(\lambda) + 1)^2 + (Im(\lambda))^2 \leq 1$. Hence, $-2 \leq Re(\lambda) \leq 0$. 
\end{proof}

\noindent  Using arguments similar to those in Lemma 2.2 of \cite{MR1918746} for the case of an undirected graph, we get the following. 

\

\begin{lemma} \label{ADinverse} Suppose $\GG$ is a connected undirected graph. The matrix $I+\A \D^{-1}$ has the following properties.
\begin{enumerate}
\item All eigenvalues of $I+\A \D^{-1}$ are real  and non-negative. 
\item $Rank(I+\A \D^{-1})$ equals $N$, if the graph $\GG$ is non-bipartite, and $N-1$ if $\GG$ is bipartite.
\item Assume that the matrix $I+\A \D^{-1}$ is diagonalizable. If the graph $\GG$ is bipartite, then the algebraic multiplicity of the eigenvalue $0$ is one.
\end{enumerate}
\end{lemma}

\begin{proof}
For the graph $\GG$, define the set $\EE^\prime$ as the set of edges in $\GG$ such that for any $(x, y) \in \mathcal{E}$, both $(x,y)$ and $(y,x)$ are in $\EE^\prime$. 
\begin{enumerate}
\item For any $x \in \mathbb{R}$, $det(xI- \A \D^{-1})=det(\D^{1/2}(xI-\D^{-1/2} \A \D^{-1/2}) \D^{-1/2})$. So, the eigenvalues of $\A \D^{-1}$ are the same as those of $\D^{-1/2}\A \D^{-1/2}$, which is symmetric. Therefore, the eigenvalues of $\A{\D}^{-1}$ and hence, those of $I+\A \D^{-1}$ are real. Since $\A \D^{-1}$ is column stochastic, the eigenvalues of $I+\A \D^{-1}$ are non-negative.

\item For $x \in \mathbb{R}^N$, \begin{eqnarray} x^{T}(\D+\A)x &=& \sum_{i=1}^{N}d_ix_i^2 +\sum_{(i,j) \in \EE^\prime}x_ix_j \nonumber\\
&=& \sum_{(i,j) \in \EE^\prime}\frac{x_i^2}{2}+ \sum_{(i,j) \in \EE^\prime}\frac{x_j^2}{2} + \sum_{(i,j) \in \EE^\prime}x_ix_j \nonumber\\
&=&\sum_{(i,j) \in \EE^\prime}\frac{(x_i+x_j)^2}{2}. \label{D+A}
\end{eqnarray}

Now, $ x^{T}(\D+\A)x =0$ implies that $x_i =-x_j$ whenever $(i,j)\in \EE^\prime$.  Let $V=\{i \in [N] : x_i>0\}, W=\{i \in [N] : x_i<0\}$ and $L= \{i \in [N] : x_i=0\}$. For a connected graph, $L$ is either empty or whole of $[N]$. 

Suppose $V$ is non-empty and suppose $v \in V$ such that $x_v = \alpha$. Then, $x_i=\alpha \forall i \in V$, and  $x_i=-\alpha \forall i \in W$. 
Thus, the graph can be partitioned into two disjoint vertex sets $V$ and $W$  such that the edges only connect vertices from $V$ to vertices of $W$,  and hence, the graph is bipartite ($V$ and $W$ being the two partitions of the graph). In this case, we have $Rank(I+\A \D^{-1})=Rank(\D+\A)=N-1$ because the solution of $ x^{T}(\D+\A)x =0$ is representable in terms of one variable alone. Since the graph cannot be partitioned in the case of a non-bipartite graph, we must have that $L=[N]$, which means that $x=\0$, where $\0$ is the zero vector of dimension $N$. $\D+\A$ is full rank in this case.

\item Since $Rank(I+\A \D^{-1})=N-1$ for a bipartite graph, the geometric multiplicity of $0$ eigenvalue is equal to one, and since the matrix is assumed to be diagonalizable, algebraic multiplicity of $0$ also equals one.
\end{enumerate}
\end{proof}

\noindent We now state and prove results for almost sure convergence of $Z_t$.

\

\begin{theorem}  \label{Friedman} For the Friedman-type reinforcement scheme, as $t \rightarrow \infty$, $Z_t \rightarrow \frac{1}{2} \1$ almost surely whenever one of the following holds.
\begin{enumerate}
\item FTSR model
 \begin{enumerate}
    \item   $0<p \leq 1$. 
    \item  $p=0$ and the graph $\mathcal{G}$ is  non-bipartite. 
   
     \end{enumerate}
\item FTNR model
\begin{enumerate}
    \item   $0 \leq p<1$. 
    \item  $p=1$ and the graph $\mathcal{G}$ is non-bipartite. 
    
     \end{enumerate}
\item The model is FTSNR, regardless of the value of $p$. 

\end{enumerate}
\end{theorem}

 \begin{proof} Using the theory of stochastic approximation, we know that it is enough to show that $\frac{1}{2} \1$ is the unique stable limit point of the solution of the O.D.E. $\dot{z}_t=h(z_t)$. From \eqref{h functions}, it is clear that $h(z)=\0$  has a unique solution in each of the cases if the matrices $(1+p)I + (1-p) \A \D^{-1}$, $p \A \D^{-1} + (1-p)(\A \D^{-1})^2 + I$ and $(pI + (1-p)\A \D^{-1})(\A+I)(I+\D)^{-1} +I $ are invertible. Since for all these matrices, each column sum is $2$, it follows that $\frac{1}{2} \1$ is a solution of $h(z)=\0$ in all three cases of Friedman-type reinforcement. Note that from Lemma~\ref{stable}, it follows that if this is a unique solution, it is also stable. Thus, it is enough to prove the invertibility of these matrices under the conditions of the theorem.  
 \begin{enumerate}
\item 
\begin{enumerate}
 \item We have $(1+p)I + (1-p)\A \D^{-1} = (1+p)( I + (1-p)/(1+p) \A \D^{-1})$. Since $\A \D^{-1}$ is a column stochastic matrix, $I + r \A \D^{-1}$ is invertible for any $r \in \mathbb{C}$ such that $ \left| r \right| < 1$. Since $ \Big \vert \frac{1-p}{1+p}\Big \vert < 1$ for $ 0 < p \leq 1$, we are done.
  \item It follows from Lemma~\ref{ADinverse} that the matrix $\D+ \A$ is invertible whenever $\GG$ is non-bipartite. 
  \end{enumerate}
\item 
\begin{enumerate}
 \item For  $0 < p < 1$, we have
 
 \begin{eqnarray*}\label{prod1}
  p \A \D^{-1} + (1-p)(\A \D^{-1})^2 + I &=& \left( I + \frac{p - \sqrt{p^2 - 4(1-p)}}{2}\A \D^{-1} \right) \times\\
   &&\left(I + \frac{p +\sqrt{p^2 - 4(1-p)}}{2}\A \D^{-1}\right)  \\
&= &  \left( I + \frac{p - \sqrt{(p+2)^2 - 8}}{2}\A \D^{-1} \right) \times\\ 
&& \left(I + \frac{p +\sqrt{(p+2)^2 - 8}}{2}\A \D^{-1}\right).
 \end{eqnarray*}
 
 Since $0<p<1$, whenever $(p+2)^2 - 8 \geq 0$, $ \Big \vert \frac{p \pm \sqrt{{{(p+2)}^2} - 8}}{2} \Big \vert < 1$. For ${(p+2)^2 }- 8 < 0$, $\Big \vert \frac{p \pm \sqrt{(p+2)^2 - 8}}{2} \Big \vert = \sqrt{1-p} < 1$. Thus,  $p \A \D^{-1} + (1-p)(\A \D^{-1})^2 + I$ can be written as a product  $(I+r_1B_1)(I+r_2B_2)$, where $B_1, B_2$ are column stochastic and $|r_1|, |r_2| < 1$. This implies $ p \A \D^{-1} + (1-p)(\A \D^{-1})^2 + I$ is invertible.

 \
 
 For $p=0$, $I+(\A\D^{-1})^2 = (I+ i \A\D^{-1})(I-i\A\D^{-1})$, and both $(I+ i \A\D^{-1})$ and $(I-i\A\D^{-1})$ are invertible since from Lemma~\ref{ADinverse} (1), we know that $\A \D^{-1}$ has real eigenvalues.

\item This case is the same as that of FTSR with $p=0$.

\end{enumerate}
\item  Notice that $(pI + (1-p)\A \D^{-1})(\A+I)(I+\D)^{-1}= [p\A +(1-p)\A\ \D^{-1 }\A + pI +(1-p)\A \D^{-1}](I+\D)^{-1}$ is irreducible for $0 \leq p \leq 1$ since  $\mathcal{G}$ is connected. To show that this matrix is aperiodic for $0 \leq p \leq 1$, it is enough to show that one of the diagonal terms of this matrix is strictly positive. This is obvious when $0 < p \leq 1$ (because of the term $pI$). For $p=0$, this is true since $\exists j \in [N]$ such that $(1,j) \in \mathcal{E}$; hence, we have $(\A \D^{-1} \A)_{1,1} >0$. Since this matrix is also stochastic, it cannot have $-1$ as an eigenvalue (because of aperiodicity, there can exist only one eigenvalue with absolute value $1$). Hence, $(pI + (1-p)\A \D^{-1})(\A+I)(I+\D)^{-1} +I$ is invertible for $0 \leq p \leq 1$.

\end{enumerate}
\end{proof}

\begin{remark}[Reinforcing a randomly chosen neighbour] \label{rem:random} Note that for the stochastic approximation argument to work, in \eqref{recur_vector}, it is enough that as $t \to \infty$, $\epsilon_t \to 0$ almost surely. Therefore, the above convergence result also holds in cases when $T_t(i)$ is not deterministic, but $\frac{T_t(i)}{t} $ converges to a fixed constant almost surely. For instance, consider the following variation in the FTNR model (with $p=1$) described above: instead of reinforcement of all the neighbours, a random neighbour is chosen and reinforced. In this case, $\{ T_t(i) \}_{i \in [N]}$ is no longer deterministic. Let $\I_{t+1}(i,j)$ denote the indicator function which takes value $1$ if the $j^{th}$ urn chooses the $i^{th}$ urn for reinforcement at time $t+1$. Since a neighbour is chosen uniformly at random, $P({\mathbbm{I}}_{t+1}(i,j)=1 | \mathcal{F}_t) = \frac{1}{d_j}$. For $ i \in [N]$, we get
\begin{eqnarray*} 
\frac{T_{t+1}(i)}{t+1} &=& \frac{T_t(i) + Cs \sum_{j \in \mathcal{N}(i)}\I_{t+1}(i,j)}{t+1} \\
&=& \frac{T_{t}(i)}{t} +  + \frac{Cs}{t+1} \left( \sum_{j \in \mathcal{N}(i)}\frac{1}{d_j} - \frac{T_{t}(i)}{Cs t} \right) +  \frac{Cs}{t+1}\sum_{j \in \mathcal{N}(i)}(  \I_{t+1}(i,j) -E[ \I_{t+1}(i,j) | \mathcal{F}_t]  ).
\end{eqnarray*}
Using stochastic approximation, we conclude that as $t \to \infty$, $\frac{T_t(i)}{t} \to T^\star(i) \coloneqq Cs \sum_{j \in \mathcal{N}(i)}\frac{1}{d_j}$ almost surely. 
We can now write a recursion for $Z_t(i)$ as follows
\begin{eqnarray*}
Z_{t+1}(i) = Z_t(i) + \frac{C}{T^\star(i) (t+1)} \left( \sum_{j \in \mathcal{N}(i)}\I_{t+1}(i,j) \chi_{t+1}(j) +\frac{(T_{t}(i)-T_{t+1}(i))Z_t(i)}{C} \right) + \epsilon_t(i),
\end{eqnarray*}


where $\epsilon_t(i) = \left( \frac{C}{T_{t+1}(i)} -   \frac{C}{T^\star(i) (t+1)} \right) \left( \sum_{j \in \mathcal{N}(i)} \I_{t+1}(i,j) \chi_{t+1}(j) + \frac{(T_{t}(i) - T_{t+1}(i)) Z_t(i)}{C} \right) \xrightarrow{t \rightarrow \infty} 0$  almost surely. Since  $E\left[(T_{t+1}(i)-T_{t}(i))| \mathcal{F}_t\right] =Cs \sum_{j \in \mathcal{N}(i)} \frac{1}{d_j}$, the stochastic approximation scheme for $Z_t$ is given by
%
%
%
$$Z_{t+1} =Z_t + \frac{1}{t+1}\left( \1 Y -Z_t({\D}^{-1}\A +Y) + \Delta M_t   \right)Y^{-1}   + \epsilon_t ,$$

where $Y$ is the diagonal matrix containing $\sum_{j \in \mathcal{N}(1)}\frac{1}{d_j}, \sum_{j \in \mathcal{N}(2)}\frac{1}{d_j} \hdots \sum_{j \in \mathcal{N}(N)}\frac{1}{d_j}$ as the diagonal entries. Thus, $Z_t  \xrightarrow{ t \rightarrow \infty}\frac{1}{2} \1$ almost surely whenever the matrix ${\D}^{-1}\A +Y$ is invertible.
\end{remark}
      \color{black}
 \

We now address the remaining cases. We use Theorem 2.1 from \cite{MR3315611} (see Theorem~\ref{Tadic_SA} in the \hyperref[appn]{Appendix}) to show that on a connected regular bipartite graph with some conditions on the initial configuration, under the reinforcement scheme FTSR with $p=0$, $Z_t$ admits an almost sure limit. Under these conditions, the $h$ function corresponding to the stochastic approximation scheme for $Z_t$ can be realized as the gradient of a function $f$. Similar results are obtained for P\'olya-type reinforcement on regular graphs. We first check that both P\'olya and Friedman-type schemes with self or neighbour reinforcement satisfy the assumptions of Theorem 2.1 \cite{MR3315611}.

 \begin{lemma} \label{assumptions}
Assume that the underlying graph $\GG$ is regular with degree $d$ and $T_{0}(i) = T_0 \ \forall i \in [N]$. Then, models with both P\'olya and Friedman-type reinforcement (except for the FTSNR model) satisfy assumptions 1-3 as stated in Theorem~\ref{Tadic_SA} in the \hyperref[appn]{Appendix}. 
\end{lemma}

\begin{proof}

Since the graph is regular, we have that $d_i=d$ and $\omega_i = \omega \ \forall \ 1 \leq i \leq N$. Then using \eqref{recur}, we have
\begin{eqnarray*}
Z_{t+1}(i) &=& \frac{T_t(i)}{T_{t+1}(i)} Z_t(i) + \frac{C}{T_{t+1}(i)} \left( \eta \chi_{t+1} (i)  +  \kappa \sum\limits_{j \in N(i)} \chi_{t+1} (j) \right) \\
&=& Z_t(i) + \frac{1}{\frac{T_0}{Cs\omega} + t+1}   \left(  \eta \frac{E[\chi_{t+1}(i) | \FF_t]}{s\omega}- Z_t(i) +  \kappa \left( \frac{\sum\limits_{j \in N(i)}E[\chi_{t+1}(j) | \FF_t]}{s\omega} \right) \right)  \\
&& +  \frac{1}{\frac{T_0}{Cs\omega} + t+1}   \left(  \frac{1}{s\omega} \left( \eta \Delta M_{t}(i) +\kappa \sum\limits_{j \in N(i)} \Delta M_{t}(j)   \right) \right).
\end{eqnarray*}

Writing in vector form, we get
\begin{eqnarray} \label{regularrecursive}
Z_{t+1} = Z_t +  \frac{1}{\frac{T_0}{Cs\omega} + t+1} \left( \frac{E[\chi_{t+1} | \FF_t]}{s\omega}  (\eta I + \kappa \A) -  Z_t + \frac{1}{s\omega} \Delta M_{t}(\eta I + \kappa \A) \right).
\end{eqnarray}
Define $X \coloneqq (1 -\frac{p\eta}{\omega})I - \frac{1}{\omega} \left( p\kappa + \frac{(1-p)\eta }{d} \right) \A - \frac{(1-p)\kappa }{d\omega} \A^2$ and $Y \coloneqq (1 + \frac{p\eta}{\omega})I + \frac{1}{\omega}\left( p\kappa  + \frac{(1-p)\eta }{d}\right) \A + \frac{(1-p)\kappa }{d\omega} \A^2$. Then, $X$ and $Y$ are symmetric matrices. For $z = (z_1, z_2 \ldots z_N) \in [0, 1]^N$, define \\ $f: \mathbb{R}^N \to \mathbb{R}^N$ by 
\begin{equation*} f(z)  =  \begin{cases}  \frac{\sum_{i = 1} ^{N} {z_i}^2 X_{ii} }{2} +  \sum_{i < j} z_i z_j X_{ij} & \text{for  P\'olya-type reinforcement scheme,}  \\
 \frac{\sum_{i = 1} ^{N} {z_i}^2 Y_{ii} }{2} +  \sum_{i < j} z_i z_j Y_{ij} -  \sum_{i = 1} ^{N} z_i &  \text{for the Friedman-type reinforcement scheme.}
\end{cases} 
\end{equation*}
Using the symmetry of $X$, for $r\in [N]$ and for  P\'olya-type reinforcement, we have

\begin{eqnarray*}
\frac{\partial f}{\partial z_r} = z_r X_{r,r} + \sum_{i <r} z_i X_{i,r} + \sum_{j>r} z_j X_{r,j} =  \sum_{j=1}^{N} z_j X_{j,r}. 
\end{eqnarray*}
Combining this with a similar calculation for Friedman-type reinforcement, we get

\begin{eqnarray} \label{X and Y} \nabla f(Z_t)  =  \begin{cases}   Z_t X
 & \text{for P\'olya-type reinforcement scheme,  }  \\
  Z_t Y -  \1 &  \text{for the Friedman-type reinforcement scheme.}
\end{cases} 
\end{eqnarray}

$X$ and $Y$ have been defined in such a way that $\nabla f(Z_t)= -h(Z_t)$  for all the six models (after substituting corresponding $\eta, \kappa$ and $\omega$ for each model). Hence, with $a_t = 1/\left(\frac{T_0}{Cs\omega} + t+1\right)$ and  $\beta_t = -\Delta M_{t}(\eta I + \kappa \A) / s\omega $  for $t \geq 0$, recursion in \eqref{regularrecursive} can be written in the form of the stochastic approximation scheme in \eqref{SA2} of the \hyperref[appn]{Appendix}.


Now we verify the three assumptions of Theorem~\ref{Tadic_SA}. Assumption $1$ is clearly satisfied for the chosen $\{ {a}_{t}\}_{t \geq 0}$. Assumption 3 is satisfied since $f$ is an analytic function on its entire domain. We now check assumption 2. Let $S_t = \sum_{i = 0} ^{t-1} a_i \gamma_i^2 \beta_{i} $, where $\gamma_i = \sum_{j=0}^{i-1} a_j$. Since $\{\Delta M_{t}\}_{t\geq 0}$is a martingale difference sequence, we have $E[S_{t+1} | {\mathcal{F}}_{t}] = S_t$. Hence, $ \{S_t\}_{t \geq 0} $ is  a martingale. Note that $a_t \sim 1/t$ and therefore $\gamma_k \sim H_k$, where $H_k$ is the $k^{th}$ Harmonic number. Since $H_n \sim \log n$, $\sum_{n=1} ^{\infty} \frac{{H_n}^4}{n^2} < \infty$. Since $\left| \left| \beta_i \right|\right|$ is uniformly bounded, we therefore have $\sum_{i = 0} ^{t} E[ { \| S_{i+1} - S_{i} \| }^2 |  {\mathcal{F}}_i] \leq  \sum_{i = 0} ^{t} {a_i}^2 {\gamma_i}^4 {\left| \left| \beta_i \right|\right|}^2 < \infty$. 
Hence, $ \{S_t\}_{t \geq 0} $ converges almost surely to a finite random vector. This implies that it is also almost surely Cauchy and we have $ \limsup\limits_{n \to \infty} \max \limits_{ n \leq k < a(n,1)} \left| \left| \sum_{i = n} ^{k} a_i {\gamma_i}^2 \beta_i \right|\right| < \infty$.
\end{proof}

 For the following theorems, recall that $m$ is the cardinality of the number of vertices in the first partition of the graph (in case when the graph is bipartite). 

\
 
 \begin{theorem} \label{Friedman2}
Consider FTSR with $p=0$ or FTNR with $p=1$ on an undirected bipartite connected regular graph. Then, as $t \rightarrow \infty$, $Z_t$ converges almost surely to a random vector $\zeta=(\zeta_1,\zeta_2 \ldots \zeta_N)$  such that $\zeta_i=\zeta_\infty \ \forall \ 1 \leq i \leq m$ and $\zeta_{i}=1-\zeta_\infty \ \forall \ m+1 \leq i \leq N$, where $\zeta_\infty$ is a random variable supported on [0,1].
\end{theorem}

\begin{proof}
The result follows from Theorem~\ref{Tadic_SA} and Lemma~\ref{assumptions}. For $p=0$, in the case of the FTSR model, we have that $\nabla f(Z_t) = Z_t \left(I + \frac{1}{d} \A \right) - \1$. 
From Lemma~\ref{ADinverse}, we know that for a bipartite graph, $Rank \left(I + \frac{1}{d} \A \right) = N-1$. So, the solution set of the linear equation $\nabla f(Z_t) = 0$ for this case can be described using one single free variable. It is straightforward to check that $ \{  (z_1, z_2 \ldots z_N) : z_i = \zeta \ \forall \ 1 \leq i \leq m $ and $z_i = 1-\zeta \ \forall \  m+1 \leq i \leq N$ for $\zeta \in [0,1]\}$ is the solution set of the equation in this case. This follows from the fact that for a bipartite graph,  $\textbf{A}$ is of the form 

$$\begin{pmatrix}
   {\bigzero}_{m \times m} & \rvline &
  \begin{matrix}
  * & \\
   & *
  \end{matrix}
    \\
\hline

\begin{matrix}
  * &  \\
   & *
  \end{matrix}
  
 & \rvline & {\bigzero}_{(N-m) \times(N-m)}
  \end{pmatrix},
$$
 where the off-diagonal blocks are nonzero, and ${\bigzero}_{m \times m}$ denotes a zero sub-matrix with $m$ rows and $m$ columns.  As remarked before, the analysis of FTNR with $p=1$ is essentially the same.  
\end{proof}
 
\

\begin{theorem} 
Consider P\'olya-type reinforcement on an undirected connected regular graph. Then, as $t \rightarrow \infty$, we have the following.
\begin{enumerate}

 \item In the following cases, $Z_t$ converges almost surely to a random vector $Z=(Z_1,Z_2 \ldots Z_N)$  such that  $Z_i=Z_\infty \ \forall \ 1 \leq i \leq N$, \color{black} where $Z_\infty$ is a random variable supported on $[0,1]$.
\begin{enumerate}
    \item PTSR model with $0 \leq p <1$.
    \item PTNR model with either $ 0 < p \leq 1$, or if $p=0$ and the graph $\GG$ is non-bipartite.
    \item PTSNR model for any $p \in [0, 1]$.
    
   \end{enumerate}
\item In the case of PTNR model with $p = 0$ and the graph $\GG$ bipartite, $Z_t$ converges almost surely to a random vector $Z^{\prime}=(Z^{\prime}_1,Z^{\prime}_2 \ldots Z^{\prime}_N)$  such that $Z^{\prime}_i= Z^{(1)}_\infty \ \forall \ 1 \leq i \leq m $ and $Z^{\prime}_{i}=Z^{(2)}_\infty \ \forall \ m +1 \leq i \leq N $, where $Z^{(1)}_\infty$ and $Z^{(2)}_\infty$ are random variables supported on $[0,1]$.

\end{enumerate}
\end{theorem}

PTSR with $p=1$ is equivalent to studying $N$ independent P\'olya-type urns where each one of them behaves like the urn studied in \cite{MR2203815}.

\begin{proof} Convergence follows from Lemma~\ref{assumptions}. We characterize the zeroes of  $\nabla f$. Note that from Lemma 13.1.1 of \cite{godsil2001}, we have that $(Rank \left( I -\frac{1}{d} {{\textbf {A}}}\right) =Rank \left( \textbf {D}-\textbf {A}\right) = N-1$ (since the graph $\GG$ is connected).

\begin{enumerate}
 \item Note that for P\'olya-type reinforcement $\nabla f (Z_t)$ is of the form $Z_t g(\A)$, where $g$ is some function of $\A$. We will first show that in all three cases, $Rank(g(\A))=N-1$. 
 \begin{enumerate}
 \item For the PTSR model, we get $\nabla f(Z_t) =  (1-p) Z_t \left(I -\frac{1}{d} \A \right)$( put $\eta =1, \kappa=0$  and $\omega =1$ in the matrix $X$ in \eqref{X and Y})  and $Rank (g(\A)) = Rank \left( (1-p) \left(I -\frac{1}{d} \A \right) \right) = N-1$. 
 
\item For the PTNR model, $\nabla f(Z_t) =  Z_t \left( dI - p\textbf{A} - \frac{1-p}{d}{\textbf{A}}^2 \right)$. For $p=1$, $\nabla f(Z_t) =  Z_t \left(dI - \A \right)$ and the result follows from the same argument as (1). For $0 < p < 1$, $dI - pA - \frac{1-p}{d}\A^2 = \left(dI - \A\right) \left(I + \frac{(1-p)}{d}\A \right)$. Since $\frac{1}{d} \A$ is stochastic, $ I + \frac{(1-p)}{d} \A$ is invertible. Hence, 
$$Rank \left(dI - p \A - \frac{1-p}{d} \A^2 \right) = \ Rank ( dI - \A) = N-1.$$ 

For $ p= 0$, when the graph is non-bipartite, by Lemma~\ref{ADinverse} (2) we have that $dI + \A$ is invertible. Therefore, $Rank ( dI - \frac{1}{d} \A^2)  = Rank \left( \left(dI + \A\right) \left(I - \frac{\A}{d} \right) \right) = Rank (I - \frac{\A}{d}) = N-1$. 

\item Finally, for the PTSNR model, we have
\begin{eqnarray*} \nabla f(Z_t)&=& Z_t \left( (d+1-p)I - p \textbf{A} -\frac{(1-p)\textbf{A}}{d} - \frac{(1- p){{\textbf{A}}^2}}{d}\right)\\
&=& Z_t \left(I-\frac{1}{d} \A \right) \left[(1-p+d)I +(1-p)\A\right].
\end{eqnarray*}
Since $Rank \left(I-\frac{1}{d} \A \right) =N-1$, it is enough to show that $(1-p+d)I +(1-p)\A$ is a full rank matrix. Since $\frac{1}{d} \A$ is stochastic and $\left|\frac{(1-p)d}{1-p+d}\right| <1$, we get that $I +\frac{(1-p)d}{(1-p+d)} \frac{1}{d}\A $ is invertible.

\end{enumerate}

In all the three cases above, since $Rank(g(\A))=N-1$, the solution set of $\nabla f=0$ can be described using a single free variable. It is easy to verify that the solution set of $\nabla f = 0$ is given by $\{  (z_1, z_2 \ldots z_N) : z_i = \zeta \ \forall \ 1 \leq i \leq N $ for $\zeta \in [0,1]\}$. 

\item For the second part of the theorem, $dI - \frac{1}{d} \A^2 = \left (dI - \A \right) \left( I + \frac{1}{d} \A \right)$. Since the graph is connected and bipartite, we have that Rank $(dI - \A) = $ Rank $\left( I + \frac{1}{d}\A \right) = N-1$.  By Sylvester's inequality, we also have that $Rank  \left( dI - \frac{1}{d} \A^2 \right) \geq N-1 + N-1 - N = N-2$. Hence, the solution set of $\nabla f = 0$ can be represented by at most two free variables. Since setting $z_i = a \ \forall \ 1 \leq i \leq m $ and $z_i = b \ \forall \ m +1 \leq i \leq N$ solves the equation, the most general solution of the equation is given by $\{  (z_1, z_2 \ldots z_N) : z_i = a \ \forall \ 1 \leq i \leq m $  and  $z_i = b \ \forall \ m +1 \leq i \leq N$ for $a, b \in [0,1]\}$. 
\end{enumerate}
\end{proof}
Thus, we have shown that $Z_t$ admits an almost sure limit, but the limiting fraction of balls of white colour in each urn depends on the underlying graph structure and the reinforcement scheme. In particular, it is crucial whether the dynamics is studied on a bipartite or a non-bipartite graph. In the next section, we study the convergence of $Z_t$ for the FTSR $p=0$ case on a bipartite graph using martingale methods that allow us to drop the regularity assumption in Theorem~\ref{Friedman2} \color{black}.


\section{Special Case: FTSR with $p=0$ on Bipartite Graphs} \label{Sec:FTSR}
We now prove a convergence result for the FTSR model with $p=0$ on bipartite graphs. The result proved in this section does not need the regularity assumption as in  Theorem~\ref{Friedman2}, but we assume the following.
\begin{itemize}
\item[(I)] $T_0(i) = T_0 \ \forall 1 \leq i \leq N$, that is, all urns have the same number of balls at time $t=0$.
\item[(II)] $\A \D^{-1}$ is  diagonalizable. More precisely, let $P$ be an $N \times N$ matrix such that $I+\A \D^{-1}=P \Lambda P^{-1}$, with $\Lambda =Diag( \lambda_1 \ldots \lambda_N)$, where  $\lambda_1, \dots, \lambda_N$ are the eigenvalues of  $I+\A \D^{-1}$.
\end{itemize}
Note that $\lambda_i \geq 0$ for all $1 \leq i \leq N$. Since the underlying graph is bipartite, $0$ is an eigenvalue, and from Lemma~\ref{ADinverse}, we know that it is a simple eigenvalue. Without loss of generality, assume that $\lambda_1=0$. Define 
\begin{equation} \label{theta}
 \theta : =\min \{\lambda_j :  \lambda_j \text{ is a non-zero eigenvalue of}\ I+\A \D^{-1}\}.
 \end{equation}
In the case of the FTSR model (I) implies that at any given time $t$, the total number of balls in each urn is the same and we denote if by $T_t$. Further, for FTSR $\omega_i=1$ for $i \in [N]$, and therefore, we have the following recursion.
\begin{equation}\label{FTSRrecur} 
Z_{t+1}= \left(1-\frac{Cs}{T_{t+1}}\right) Z_t + \frac{C}{T_{t+1}}\chi_{t+1}.
\end{equation}

 We show that the fraction of balls of either colour in urns placed on vertices of $V$ (resp. $W$) converges to the same random limit almost surely.   Define $Z^v_t = (Z_t(1),Z_t(2)\ldots Z_t(m),0,0\ldots0)$ and  $Z^w_t = (Z_t(m+1),Z_t(m+2)\ldots Z_t(N),0,0\ldots0)$ as  $N$-dimensional vectors. For $ 0 <k <N$, let  $e_k \in \mathbb{R}^N$ be a row vector with $1$ as the entry in the first $k$ positions and $0$ otherwise (in the rest of the $N-k$ positions). Further, define
$$\bar{d_v}= \sum\limits_{i \in V} d_i = e_m \D e_m^\top, \bar{ d_w}=\sum\limits_{i \in W} d_i = (\1-e_m) \D (\1-e_m)^\top, \bar{d} =\sum\limits_{i \in [N]} d_i = \1 D \1^\top, $$

$$\bar{Z}_t =\dfrac{1}{\bar{d}} \sum\limits_{i \in [N]} d_i Z_t(i) = \dfrac{1}{\bar{d}} Z_t \D \1^\top, \ \bar{Z}^v_t=\dfrac{1}{\bar{d_v}} \sum\limits_{i \in V} d_i Z_t(i) =  \dfrac{1}{\bar{d_v}} Z_t \D e_m^\top $$

and

$$ \bar{Z}^w_t=\dfrac{1}{\bar{d_w}} \sum\limits_{i \in W} d_i Z_t(i) =  \dfrac{1}{\bar{d_w}} Z_t \D (\1-e_m)^\top.$$

The main result of this section stated in the following theorem shows that urns on vertices in $V$ and $W$ synchronize almost surely, respectively. More precisely, we show that as $t \to \infty$,  $Z_t(i)-\bar{Z}^v_t \xrightarrow{a.s.} 0  $ for $i \in V$, $Z_t(i)-\bar{Z}^w_t \xrightarrow{a.s.}0  $ for $i \in W$, and that $\bar{Z}_t^v+\bar{Z}_t^w \xrightarrow{a.s.} 1$.

\

\begin{theorem} \label{syncFTSR} For the FTSR model with $p=0$ on a bipartite graph, under assumptions (I) and (II), as $t \to \infty$, $Z^v_t - \bar{Z}^v_t e_m \xrightarrow{a.s.} \0$ and  $Z^w_t - \bar{Z}^w_t e_{N-m} \xrightarrow{a.s.} \0$.  
\end{theorem}

\

\noindent We first show that $\bar{Z}_t$ converges to a unique deterministic limit. 

\

\begin{theorem} For the FTSR model with $p=0$ on a bipartite graph, under assumption (I), $\bar{Z}_t \xrightarrow{a.s.} 1/2$ and $\bar{Z}_t^v+\bar{Z}_t^w \xrightarrow{a.s.} 1$, as $t \to \infty$.
\end{theorem}

\begin{proof}
We write a stochastic approximation scheme for  $\bar{Z}_t$. Using \eqref{FTSRrecur},
\begin{eqnarray}\label{avgZ_recur}
\bar{Z}_{t+1} &=& \frac{1}{\bar{d}} Z_{t+1} \D \1^\top \nonumber \\
&=&  \frac{1}{\bar{d}} \left(1-\frac{Cs}{T_{t+1}}\right) Z_t \D \1^\top + \frac{C}{\bar{d} ~ T_{t+1}}\chi_{t+1} \D \1^\top \nonumber \\
&=& \bar{Z}_t + \frac{C}{\bar{d} T_{t+1}} \left( \chi_{t+1} - s Z_t \right) \D \1^\top  \nonumber \\
&=& \bar{Z}_t + \frac{C}{\bar{d} ~ T_{t+1}} \left( E[\chi_{t+1} | \FF_t] - s Z_t \right) \D \1^\top +  \frac{C}{\bar{d} ~ T_{t+1}} \Delta \bar{M}_{t+1},
\end{eqnarray}
where $\Delta \bar{M}_{t+1}= (\chi_{t+1} - E[\chi_{t+1} | \FF_t])\D \1^\top$ is a bounded martingale difference. Note that
\begin{eqnarray} \label{avg_condexp}
\left( E[\chi_{t+1} | \FF_t] - s Z_t \right) \D \1^\top &=& s(\1-Z_t \A\D^{-1}) \D \1^\top - \bar{d}s \bar{Z}_t \nonumber \\
&=& s ( \bar{d} - Z_t \A \1^\top - \bar{d} \bar{Z}_t) \nonumber \\
&=& s\bar{d} (1-2 \bar{Z}_t),
\end{eqnarray}
since $\A \1^\top = \D \1^\top $. From the theory of stochastic approximation, we know that the iterates of \eqref{avgZ_recur} converge almost surely to the stable zeroes of $\dot{x}_t =  s\bar{d} (1-2 x_t)$. Thus, $\bar{Z}_t \to 1/2$ almost surely. Since the graph is bipartite, we have that $\bar{d_v}=\bar{d_w}= \frac{\bar{d}}{2}$. Hence, $\bar{Z}_t^v+\bar{Z}_t^w=2\bar{Z}_{t}  \to 1$ almost surely.
\end{proof}

 We now prove a useful lemma for the moments of certain linear functions of $\{ Z_t(i) \}_{i \in [N]}$. Let $\mathcal{M}_{N, m}$ be the space of all $N \times N$ matrices with the property that for each column of the matrix, the sum of the first $m$ terms of that column is equal to the sum of the last $N-m$ terms. That is, 
$$ \mathcal{M}_{N, m}  = \{ M \in \mathbb{R}^{N \times N} : e_m M= (\1-e_m) M \}, $$
where $\mathbb{R}^{N \times N}$ denotes the set of all $N \times N$ matrices over $\mathbb{R}$. 

\

\begin{lemma}\label{M_N} Let $\Phi_{t} =Z_t Q$ such that $Q \in \mathcal{M}_{N, m}$.  The following hold under assumptions (I) and (II).
\begin{enumerate}
\item Suppose that, in addition, $Q$ is such that $e_m Q=\0$. Then, $E[\Phi_{t}] \sim \frac{1}{t^{\theta}}$, where $\theta$ is as defined in \eqref{theta}.

\item $Var(\Phi_{t}) \sim \frac{1}{t^\epsilon} $ for some $\epsilon > 0$. 
\end{enumerate}
\end{lemma}

\begin{proof}
\begin{enumerate}

\item From \eqref{FTSRrecur}, we get
\begin{eqnarray}\label{conditionalexpectation}
E[Z_{t+1} | \FF_t] &=& Z_t\left(1-\frac{Cs}{T_{t+1}}\right) +\frac{Cs(\1-Z_t \A \D^{-1})}{T_{t+1}} \nonumber \\
&=& Z_t\left(I-\frac{Cs}{T_{t+1}}(I+\A \D^{-1})\right) +\frac{Cs\1}{T_{t+1}}.
\end{eqnarray}
This implies
\begin{equation} \label{Exp_Phi}
E[\Phi_{t+1} ] = Z_{0}\left(\prod_{i=0}^{t} H_i \right) Q  + \sum_{i=0}^{t} \frac{Cs\1}{T_{i+1}}  \left(\prod_{j=i+1}^{t} H_j \right) Q,
\end{equation}
where $H_j= I-\frac{Cs}{T_{j+1}}(I+\A \D^{-1})$ for $ j\geq 1$. Now, using $\1 \A \D^{-1} = \1$, we get that the second term in \eqref{Exp_Phi} is zero. Indeed,
\begin{eqnarray*} 
\frac{Cs \1}{T_{i+1}} \left[ \prod_{j=i+1}^t  \left(I-\frac{Cs}{T_{j+1}} (I + \A \D^{-1} ) \right)\right] Q &=& \frac{Cs}{T_{i+1}}\left(\1-\frac{Cs}{T_{i+2}} (\1 + \1) \right) \left[\prod_{j=i+2}^t \left(I-\frac{Cs}{T_{j+1}} (I + \A \D^{-1} ) \right)\right] Q \nonumber \\
&=& \frac{Cs}{T_{i+1}}  \left(1-\frac{2Cs}{T_{i+2}}  \right) \1 \left[ \prod_{j=i+2}^t \left(I-\frac{Cs}{T_{j+1}} (I + \A \D^{-1} ) \right)\right] Q \nonumber \\
&=& \frac{Cs}{T_{i+1}} \left[ \prod_{j=i+1}^t   \left(1-\frac{2Cs}{T_{j+1}}  \right)\right] \1 Q \\
&=& \0. 
\end{eqnarray*}
Using the diagonalization of $I+\A \D^{-1}$, the first term in \eqref{Exp_Phi} can be rewritten as \\ $Z_{0} P \left[ \prod_{i=0}^{t}  \left(I-\frac{Cs}{T_{i+1}}\Lambda\right)\right] P^{-1} Q $. Using Euler's approximation, we have that, for $ 2 \leq j \leq N$, $\prod_{i=0}^{t} \left(1-\frac{Cs}{T_{i+1}}\lambda_j\right) \sim t^{-\lambda_j}$. Hence, as $t \rightarrow \infty$
\begin{equation} 
\prod_{i=0}^{t} \left(I-\frac{Cs}{T_{i+1}}\Lambda\right) \to e_1^\top e_1 = \begin{pmatrix}
    1 & 0 & 0 & \hdots & 0 \\
    0& 0 & 0 & \hdots & 0\\
    \vdots &  \vdots & \vdots &  \vdots  & \vdots \\
    0 & 0 & 0 & \hdots & 0 \\
   \end{pmatrix}.
  \end{equation}

Since the first row of $P^{-1}$ is the left eigenvector of the matrix $I+\A \D^{-1}$ corresponding to the eigenvalue $0$ and is of the form $\alpha(\1-2e_m)$ for some $\alpha \in \mathbb{R}$, we get that
\begin{equation}   \label{Exp_firstterm_limit}
e_1^\top e_1 P^{-1} Q = \0_{N \times N},
  \end{equation}
because $Q \in \mathcal{M}_{N, m}$. Thus, $E[\Phi_t] \to \0$ as $t \to \infty$, with rate $\frac{1}{t^{\theta}}$, (since  $\prod_{i=0}^{t} \left(1-\frac{Cs}{T_{i+1}}\lambda_j\right) \sim t^{-\lambda_j}$ for $j \geq 2$).

\

\item

From \eqref{FTSRrecur}, we get
\begin{equation*}
Var(Z_{t+1}| {\mathcal{F}}_t) =\frac{C^{2}}{T^2_{t+1}}Var(\chi_{t+1}| {\mathcal{F}}_t).
\end{equation*}
Note that for $1 \leq j \leq N$, $ E[Var(\chi_{t+1}(j)| {\mathcal{F}}_t)]= E[Var(Y_{t+1}(j)| {\mathcal{F}}_t)]$, and using the conditional distribution of $Y_{t+1}(j)$ from \eqref{Y} we get
 \begin{eqnarray} 
 E[Var(Y_{t+1}(j)| {\mathcal{F}}_t)]&= &\frac{s(s-1)\sum_{l \in N(j)}E[{Z_t(l)}^2] +  s\sum_{l \in N(j)}E[{Z_t(l)}]}{d_j} \nonumber \\
 &&- \left( \frac{s\sum_{l \in N(j)}E[{Z_t(l)}]}{d_j} \right)^{2} \eqqcolon \gamma_t(j). \label{gamma_t}
 \end{eqnarray}


Thus, we get
\begin{equation} \label{cond_Var}
 E[Var(Z_{t+1}| {\mathcal{F}}_t)] =\frac{C^{2}}{{T_{t+1}}^2}
 \begin{pmatrix}
    \gamma_t(1) & & \\
    & \ddots & \\
    & & \gamma_t(N)
  \end{pmatrix} \eqqcolon \frac{C^{2}}{{T_{t+1}}^2} G_t.
 \end{equation}

From \eqref{conditionalexpectation} and \eqref{cond_Var}, we get
\begin{eqnarray*}
Var(Z_{t+1}) &=& E[Var(Z_{t+1}| {\mathcal{F}}_t)]  + Var(E[Z_{t+1} | \FF_t]) \\
&=&  \frac{C^{2}}{{T_{t+1}}^2} G_t + H_t^\top Var(Z_t)H_t,
\end{eqnarray*}
where $H_t= I-\frac{Cs}{T_{t+1}}(I+\A \D^{-1})$. We have
\begin{eqnarray*} 
Var(\Phi_{t+1}) &=& Q^\top Var(Z_{t+1}) Q \\
&=& Q^\top \left(  \frac{C^{2}}{{T_{t+1}}^2} G_t + H_t^\top Var(Z_t)H_t \right) Q \\
&=& Q^\top \left[ \left( \prod_{i=0}^{t-1} H_{t-i}^\top \right) Var(Z_1) \left( \prod_{i=1}^t H_i \right) + \sum\limits_{k=1}^t \frac{C^2}{T_{k+1}^2} \left( \prod_{i=0}^{t-k-1} H_{t-i}^\top \right) G_k \left( \prod_{i=k+1}^t H_i   \right) \right] Q.\\
\end{eqnarray*}

Since $Var(Z_1) = \frac{C^2}{T_1^2} G(0)$, we have

\begin{equation*} 
Var(\Phi_{t+1})= Q^\top\left[ \sum\limits_{k=0}^t \frac{C^2}{T_{k+1}^2} \left( \prod_{i=0}^{t-k-1} H_{t-i}^\top \right) G_k \left( \prod_{i=k+1}^t H_i   \right) \right] Q. 
\end{equation*}
Using diagonalization of $I+\A \D^{-1}$, we get
\begin{eqnarray*} 
Var(\Phi_{t+1}) &=& Q^\top\left[ \sum\limits_{k=0}^t \frac{C^2}{T_{k+1}^2}  (P^{-1})^\top \left( \prod_{i=0}^{t-k-1} \left( I-\frac{Cs}{T_{t-i+1}} \Lambda  \right) \right) P^\top G_k P \prod_{i=k+1}^t \left( I-\frac{Cs}{T_{i+1}} \Lambda  \right) P^{-1} \right] Q \\
& =& Q^\top (P^{-1})^\top \Delta_t P^{-1} Q,
\end{eqnarray*}
where $\Delta_t \coloneqq \sum\limits_{k=0}^t \frac{C^2}{T_{k+1}^2}  \left[ \prod_{i=0}^{t-k-1} \left( I-\frac{Cs}{T_{t-i+1}} \Lambda  \right) \right] P^\top G_k P \prod_{i=k+1}^t \left( I-\frac{Cs}{T_{i+1}} \Lambda  \right) $. 

For $(p, q) \neq (1, 1)$, the $(p, q)^{th}$ element of $\Delta_t$ is given by
\begin{eqnarray*}\Delta_{t}(p,q)&=& \sum\limits_{k=0}^t \frac{C^2}{T_{k+1}^2}  \left[ \prod_{i=0}^{t-k-1} \left( 1-\frac{Cs}{T_{t-i+1}} \lambda_p  \right) \right] s_k(p, q) \prod_{i=k+1}^t \left( 1-\frac{Cs}{T_{i+1}} \lambda_q  \right),
\end{eqnarray*}
where $s_k(p, q)$ is the $(p, q)^{th}$ element of $P^\top G_k P$. Again, by using the Euler's approximation, we get
\begin{eqnarray} \label{rates}
\Delta_{t}(p,q) & \leq & C^\prime t^{-(\lambda_p+\lambda_q)} \sum\limits_{k=1}^t \frac{1}{k^{2-(\lambda_p+\lambda_q)}} \nonumber \\
&=& \begin{cases}  
\mathcal{O} \left( \frac{1}{t} \right)  & \text{ if } \lambda_p + \lambda_q >1,  \\
\mathcal{O} \left( \frac{\log t}{t} \right)  & \text{ if } \lambda_p + \lambda_q = 1,  \\
 \mathcal{O} \left( t^{-( \lambda_p + \lambda_q)} \right) & \text{ if } \lambda_p + \lambda_q <1.
\end{cases}
\end{eqnarray}
Since $\Delta_{t}(1,1)$ converges to $\sum\limits_{k=0}^\infty \frac{C^2}{T_{k+1}^2}  s_k(p, q) $, there exists a constant $c$ such that $\Delta_{t} \to c e_1^\top e_1$ as $t \to \infty$.  Since $e_1^\top e_1 P^{-1} Q = \0_{N \times N}$, $Var(\Phi_{t+1})$ converges to zero at rates given in \eqref{rates}. This concludes the proof. 
\end{enumerate}
\end{proof}
We remark that the expression for $\gamma_t(j)$ as defined in \eqref{gamma_t} in the proof above is different in case of sampling without replacement. However, since it does not affect the final result, we skip the details for the sampling without replacement case and do not state the explicit expression of $\gamma_t(j)$ for that case.

\

Before we show that the urns synchronize almost surely on each partition, we will show that the fraction of balls of either colour in each urn attains an almost sure limit. Using the convergence of quasi-martingales (see \cite{Metivier+1982}), we get the following. 

\

\begin{proposition}\label{proposition} $Z_t$ admits an almost sure limit.
\end{proposition}
\begin{proof}
From \eqref{avgZ_recur} and \eqref{avg_condexp}, we get $E[\bar{Z}_{t+1} | \FF_t] = \bar{Z}_t + \frac{2Cs}{T_{t+1}}  (1/2 - \bar{Z}_t)$. This implies
\begin{eqnarray*}
E[\bar{Z}_{t+1}-1/2] &=& E[\bar{Z}_t] - 1/2 + \frac{2Cs}{T_{t+1}}  (1/2 - E[\bar{Z}_t]) \\\nonumber
&=& \prod_{k=0}^t \left( 1 - \frac{2Cs}{T_{k+1}} \right)(\bar{Z}_0-1/2) \\ \nonumber
& \sim & \frac{1}{t^2} (\bar{Z}_0-1/2).
\end{eqnarray*}

Hence, \begin{equation}
\label{expdecay}\left|E[\bar{Z}_{t}-1/2]\right| \sim  \frac{1}{t^2} (\left|Z_0-1/2\right|). 
\end{equation}

Using \eqref{conditionalexpectation}, we have
\begin{eqnarray*} \sum_{t=0}^{\infty} E[\|E[Z_{t+1} | \FF_t]-Z_t\|] & = & Cs \sum_{t=0}^{\infty} \frac{ E[\|\1-Z_t(I+\A \D^{-1})\|]}{T_{t+1}}\\
&=& Cs \sum_{t=0}^{\infty} \frac{ E[\|\1-\bar{Z}_{t}\1(I+\A \D^{-1})-(Z_t-\bar{Z}_{t}\1)(I+\A \D^{-1})\|]}{T_{t+1}}\\
&\leq & Cs \sum_{t=0}^{\infty} \left\{ \frac{ E[\|\1-2\bar{Z}_{t}\1\|+  E[\| (Z_t- \bar{Z}_{t}\1)(I+\A \D^{-1})\|]}{T_{t+1}} \right\} \\
&\leq & Cs  \left\{ \sum_{t=0}^{\infty} \frac{ E[\|\1-2\bar{Z}_{t}\1\|}{T_{t+1}} + \sum_{t=0}^{\infty} \frac{ E[\|(Z_t- \bar{Z}_{t}\1)(I+\A \D^{-1})\|]}{T_{t+1}} \right\} \\
&\leq & Cs  \left\{ \sum_{t=0}^{\infty} \frac{2\|1\| E[|\bar{Z}_{t}-1/2|]}{T_{t+1}} + \sum_{t=0}^{\infty} \frac{ E[\| (\bar{Z}_{t}\1-Z_t)(I+\A \D^{-1})\|]}{T_{t+1}} \right\} \\
&\leq & Cs  \left\{2\|1\|\left[ \sum_{t=0}^{\infty} \frac{ E[|\bar{Z}_{t}-E[\bar{Z}_{t}]|]}{T_{t+1}} + \sum_{t=0}^{\infty} \frac{ E[\vert E[\bar{Z}_{t}-1/2] \vert]}{T_{t+1}} \right]\right.\\
&& +  \sum_{t=0}^{\infty} \frac{ E[\| (\bar{Z}_{t}\1-Z_t -E[\bar{Z}_{t}\1-Z_t])(I+\A \D^{-1})\|]}{T_{t+1}} \\
&& + \left. \sum_{t=0}^{\infty} \frac{ E[\| E[(\bar{Z}_{t}\1-Z_t)(I+\A \D^{-1})] \|]}{T_{t+1}} \right\} \\
&\leq & Cs  \left\{2\|1\|\left[ \sum_{t=0}^{\infty} \frac{ E[|\bar{Z}_{t}-E[\bar{Z}_{t}]|]}{T_{t+1}} + \sum_{t=0}^{\infty} \frac{ \vert E[\bar{Z}_{t}-1/2] \vert}{T_{t+1}} \right]\right.\\
&&+   \sum_{t=0}^{\infty} \frac{ E[\| (\bar{Z}_{t}\1-Z_t -E[\bar{Z}_{t}\1-Z_t])(I+\A \D^{-1})\|]}{T_{t+1}} \\
&&+ \left. \sum_{t=0}^{\infty} \frac{ \| E[(\bar{Z}_{t}\1-Z_t)(I+\A \D^{-1})] \|}{T_{t+1}} \right\} \\
&\leq & Cs  \left\{2\|1\|\left[ \sum_{t=0}^{\infty} \frac{ \sqrt{Var(\bar{Z}_{t})}}{T_{t+1}}  +\sum_{t=0}^{\infty} \frac{ |E[\bar{Z}_{t}-1/2]|}{T_{t+1}} \right]  \right. \\
&& +   \sum_{t=0}^{\infty} \frac{ \sqrt{\Tr(Var((\bar{Z}_{t}\1-Z_t)(I+\A \D^{-1})))}}{T_{t+1}}\\
&&+ \left.\sum_{t=0}^{\infty} \frac{ \| E[(\bar{Z}_{t}\1-Z_t)(I+\A \D^{-1})]\|}{T_{t+1}} \right\}. 
\end{eqnarray*}
Note that $\bar{Z}_{t}= \frac{Z_tD {\1}^{T}}{\bar{d}}$and $(\bar{Z}_{t}\1-Z_t)(I+\A \D^{-1}) = Z_t \left(\frac{D {\1}^{T}\1}{\bar{d}}-I \right) (I+\A \D^{-1}) = Z_t \left[ \frac{2D {\1}^{T}\1}{\bar{d}} -I-\A \D^{-1}\right]$. Since $\frac{2D {\1}^{T}\1}{\bar{d}}-( I + \A \D^{-1}),\frac{D {\1}^{T}\1}{\bar{d}} \in \mathcal{M}_{N, m}$ and $e_m \left(\frac{2D {\1}^{T}\1}{\bar{d}}-( I + \A \D^{-1})\right)  = \frac{2d_v \1}{\bar{d}} -e_m -(\1-e_m) = \0$\\( since $d_v=\frac{\bar{d}}{2}$), therefore by Lemma~\ref{M_N} we get that 
$$\sum_{t=0}^{\infty} \frac{ \| (E[\bar{Z}_{t}\1-Z_t])(I+\A \D^{-1})\|}{T_{t+1}}  <\infty, \sum_{t=0}^{\infty} \frac{ \sqrt{Var(\bar{Z}_{t})}}{T_{t+1}} < \infty  \text{ and }$$  
$$\sum_{t=0}^{\infty} \frac{ \sqrt{\Tr(Var((\bar{Z}_{t}\1-Z_t)(I+\A \D^{-1})))}}{T_{t+1}} <\infty.$$

Using \eqref{expdecay}  we also have that $\sum_{t=0}^{\infty} \frac{ |E[\bar{Z}_{t}-1/2]|}{T_{t+1}} <\infty$. Thus, $\sum_{t=0}^{\infty} E[\|E[Z_{t+1} | \FF_t]-Z_t\|] < \infty$. That is, $\{Z_t\}_{t\geq 0}$ is a quasi-martingale, and therefore, $Z_t$ has an almost sure limit.
\end{proof}

We are now ready to prove the main synchronization result for FTSR on bipartite graphs. Let $I_m$ be the diagonal matrix containing $1$ at the first $m$ diagonal positions, and $0$ elsewhere. 
\begin{proof}[Proof of Theorem~\ref{syncFTSR}]
Let $\Phi_t^v=Z_t^{v}-\bar{Z}^t_v e_m =Z_tM$ where $M=I_m- \frac{De_m^{T}e_m}{\bar{d_v}}$. Note that $e_m M=e_m\left(I_m- \frac{De_m^{T}e_m}{\bar{d_v}}\right)=e_m-e_m=\0$, and $M \in \mathcal{M}_{N, m}$. Using Lemma~\ref{M_N}, we get that
 $Var(\Phi_t^v)\rightarrow \0$ and  $E[\Phi_t^v] \rightarrow \0$, which implies that $Z^v_t - \bar{Z}^v_t e_m \xrightarrow{\mathcal{L}^2} \0$. Since $Z_t$ is a bounded random vector and since we showed above that it has an almost sure limit, we have that $Z_t \xrightarrow{a.s}Z_{\infty}$, where $Z_{\infty}$ is finite almost surely. Since ${\|\cdot \|}^{2}$ is a continuous function on $ \mathbb{R}^{N}$, using continuous mapping theorem we have that $\| \Phi_t^v \|^{2} \xrightarrow{a.s} {\|Z_{\infty}M\|}^{2}$. But we already know that $E[\|\Phi_t^v \|^{2}] \rightarrow 0$; therefore,  $E[{\|Z_{\infty}M\|}^{2}] =0$  (by bounded convergence theorem). This means that $ Z_{\infty}M =\0$  almost surely. That is, $\Phi_t^v=Z^v_t - \bar{Z}^v_t e_m \xrightarrow{a.s.} \0$. Similarly,  $Z^w_t - \bar{Z}^w_t e_{N-m} \xrightarrow{a.s.} \0$. \color{black}
\end{proof}

 We remark that the PTSR case can also be treated in the same way and the condition of regularity can be dropped. In particular, the following synchronization result holds.

\

\begin{theorem}  \label{syncPTSR} For the PTSR model with $0 \leq p <1$ under assumptions (I) and (II), $Z_t - \bar{Z_t} \1 \xrightarrow{a.s.} \0$ as $t \to \infty$. 
\end{theorem}
The proof involves the same arguments as above. Recalling \eqref{FTSRrecur} and using $E[\chi_{t+1} | \FF_t] =s Z_t(p I +(1-p)\A \D^{-1}) $ for PTSR model, we get
\begin{eqnarray*}
E[Z_{t+1} | \FF_t] &=& Z_t\left(1-\frac{Cs}{T_{t+1}}\right) +\frac{CE[\chi_{t+1} | \FF_t]}{T_{t+1}} \nonumber \\
&=& Z_t\left(I-\frac{Cs(1-p)}{T_{t+1}}(I-\A \D^{-1})\right),
\end{eqnarray*}

which is a much simpler recursion compared to \eqref{conditionalexpectation} above, and hence, a result analogous to Lemma~\ref{M_N} can be proved with ease, using similar arguments as above. We can also show that $Z_t$ admits an almost sure limit using the quasi-martingale argument. The rest of the proof is similar.


\color{black}


\section{Fluctuation Results}\label{Sec:Fluc}
We now state and prove a fluctuation result for FTSR and FTNR models on a specific class of graphs. We also assume that the sampling is done with replacement. 

\

\begin{theorem} Consider the Friedman-type reinforcement scheme, and assume that the sampling is done with replacement. Assume that $Z_t$ has a unique convergent point  $\frac{1}{2} \1$.  Let $\rho$ be the smallest eigenvalue of the matrix $p \A\D^{-1} + (1-p)(\A\D^{-1})^2 + I$  for the FTNR case and the smallest eigenvalue of $(1+p)I+(1-p)\A\D^{-1}$ for the FTSR case. Further, assume that the matrix $\A\D^{-1}$ is symmetric, with $\A\D^{-1} =U \Lambda' {U}^{-1}$. Then, we have
\begin{enumerate} 
\item When $\rho >\frac{1}{2}$, $\sqrt{t} \left(Z_t - \frac{1}{2} \1\right) \xrightarrow{d} \mathcal{N}\left(\0, \Sigma\right)$, 
where $ \Sigma$ equals $\frac{1}{4s}I + \frac{p}{2s} \A\D^{-1} + \frac{1-p}{2s}(\A\D^{-1})^2 $ for the FTNR case and equals $\frac{1+2p}{4s}I +\frac{1-p}{2s}\A\D^{-1}$ for the FTSR case. 

\item Let $\lambda'_1,\lambda'_2 \ldots \lambda'_N $ be the eigenvalues of $I+2p\A\D^{-1} +2(1-p)(\A\D^{-1})^2$ and $\mu_1, \mu_2 \ldots \mu_N$ be the eigenvalues of $(2p+1)I + 2(1-p)\A\D^{-1}$. Let $a_1, a_2 \ldots a_q$ and $b_1, b_2 \ldots b_r$ be indices such that $\lambda'_{a_1} = \lambda'_{a_2} = \ldots =\lambda'_{a_q}= \mu_{b_1} =\mu_{b_2}= \ldots =\mu_{b_r} =0$. When $\rho=\frac{1}{2}$, we have
\begin{equation*}
\sqrt{\frac{t}{\log t}} \left(Z_t - \frac{1}{2} \1\right) \xrightarrow{d} \mathcal{N}\left(\0, \widetilde{\Sigma}\right), 
\end{equation*}
where $\widetilde{\Sigma} =\frac{1}{4s}U B U^{-1}$ where the matrix $B$ is a diagonal matrix, defined as follows for FTSR and FTNR cases. For the FTNR case, $B$ is such that $B_{a_1,a_1}=B_{a_2,a_2} =\ldots =B_{a_q,a_q}=1$. For the FTSR case, it is such that $B_{b_1,b_1}=B_{b_2,b_2} =\ldots =B_{b_r, b_r}=1$. The rest of the diagonal terms are zero in both cases.
\end{enumerate}
\end{theorem}

\

\begin{remark} 
From Theorem~\ref{Friedman}, we know the exact conditions for $Z_t$ to have the unique limit $\frac{1}{2} \1$. Also note that using an argument similar to that in Lemma~\ref{ADinverse}, we get that the eigenvalues of matrices  $(1+p)I+(1-p)\A\D^{-1}$ and $p \A\D^{-1} + (1-p)(\A\D^{-1})^2 + I$ are all real. Further, note that $\Lambda^\prime = \Lambda - I$, where $\Lambda = Diag (\lambda_1, \dots, \lambda_N)$ is as defined in the earlier sections. 
\end{remark}

\begin{proof} We use Theorems 1.1 and 2.1 from \cite{MR3582813}. From Lemma~\ref{stable}, recall that 
\begin{equation*} \frac{\partial h}{\partial z}=
   \begin{cases}   -(I +p\A\D^{-1} +(1-p)(\A\D^{-1})^2) & \ \text{for the FTNR case, } \\
                    -((1+p)I +(1-p)\A\D^{-1}) & \ \text{for the FTSR case.}\\
\end{cases}
\end{equation*}

Define $\Gamma \coloneqq \lim\limits_{t \to \infty} \frac{1}{s^2} E [{\Delta M_{t+1}}^{T}\Delta M_{t+1} |{\mathcal{F}}_{t}]$, where $\{\Delta M_{t}\}_{t\geq 0}$ is the martingale difference sequence as defined in Section~\ref{Sec:prelim} . Note that $\Gamma$ is the same for both FTSR and FTNR models since it only depends on sampling, and not reinforcement.

Now, for $i \neq j$, since the event of sampling from urn $i$ is independent of sampling from urn $j$, we have

\begin{equation*}
(E[{\Delta M_{t+1}}^{T}\Delta M_{t+1} | {\mathcal{F}}_t])_{i,j} = E[(\chi_{t+1}(i)-   E [\chi_{t+1}(i) |{\mathcal{F}}_t])(\chi_{t+1}(j)-   E [\chi_{t+1}(j) |{\mathcal{F}}_t])|{\mathcal{F}}_t]=0.
\end{equation*}
For $1 \leq i \leq N$,
 \begin{align} \label{Gamma} 
 \lim_{t \to \infty}  (E[{\Delta M_{t+1}}^{T}\Delta M_{t+1} | \mathcal{F}_t])_{i,i} &=& \lim_{t \to \infty} E[ (Y_{t+1}(i))^2| {\mathcal{F}}_t] + s^2 \lim_{t \to \infty} \left(  p Z_t(i) + \frac{1-p}{d_i} \sum\limits_{l \in N(i)}Z_t(l)\right)^2  \nonumber \\
 &&  - 2s \lim_{t \to \infty}\left(  p Z_t(i) + \frac{1-p}{d_i} \sum\limits_{l \in N(i)} Z_t(l)\right) \times \lim_{t \to \infty} E[Y_{t+1}(i)| {\mathcal{F}}_t].
\end{align}
We have $\lim\limits_{t \to \infty} E[Y_{t+1}(i) | {\mathcal{F}}_t] = \lim\limits_{t \to \infty} s \left(  p Z_t(i) + \frac{1-p}{d_i} \sum\limits_{l \in N(i)} Z_t(l)\right) = \frac{s}{2}$. Further,
\begin{align*} E[(Y_{t+1}(i))^2 | {\mathcal{F}}_t] & = \left\{s(s-1)(Z_t(i)^2 + sZ_t(i) \right\} p  + \frac{1-p}{d_i}  \left\{ s(s-1)\sum\limits_{l \in N(i)} (Z_t(l))^2 + s \sum\limits_{l \in N(i)} Z_t(l) \right\}. \\
\end{align*}
This gives $\lim_{t \to \infty}E[(Y_{t+1}(i))^2 | {\mathcal{F}}_t ] = \frac{s(s+1)}{4}$. Now, substituting these in \eqref{Gamma}, we get 
\begin{equation*} 
\lim_{t \to \infty} \left( E[{\Delta M_{t+1}}^{T}\Delta M_{t+1} | {\mathcal{F}}_t] \right)_{i,i} = \frac{s(s+1)}{4} + \frac{s^2}{4} - 2s\left(\frac{s}{2}\right) \left(\frac{1}{2}\right) = \frac{s}{4}. 
\end{equation*}
Thus,
 \begin{equation*} 
 \Gamma = \frac{s}{4s^2} I_{N \times N} = \frac{1}{4s} I_{N \times N}.
 \end{equation*}
For the FTNR case, using $\A \D^{-1}=U{\Lambda}' {U}^{-1}$ and the symmetry of $\A\D^{-1}$, for $\rho > 1/2$ we get
\begin{eqnarray*} 
\Sigma &=& \int_{0}^{\infty} {(e^{-(-  \frac{\partial h}{\partial z}(\frac{1}{2}\1)- \frac{I}{2})u})}^{T}    \Gamma e^{-(-  \frac{\partial h}{\partial z}(\frac{1}{2}\1)- \frac{I}{2})u } du \\
&=& \frac{1}{4s}\int_{0}^{\infty} {(e^{-(p\A\D^{-1} + (1-p)(\A\D^{-1})^2 + \frac{I}{2} )u})^{T} (e^{-(p\A\D^{-1} + (1-p)(\A\D^{-1})^2 + \frac{I}{2} )u}) } \,du \ \nonumber\\
&=&\frac{1}{4s}\int_{0}^{\infty} U{(e^{-(p\Lambda' + (1-p)(\Lambda')^2 + \frac{I}{2} )u}){U}^{-1}{U} (e^{-(p\Lambda' + (1-p)(\Lambda')^2 + \frac{I}{2} )u}) }{U}^{-1} \,du \ \nonumber\\
&=&\frac{1}{4s} U \int_{0}^{\infty} {(e^{-(2p \Lambda'  + 2(1-p){\Lambda' }^2 + I )u})}\,du \ \ {U}^{-1} \\
& =&\frac{1}{4s} U(2p \Lambda'  + 2(1-p){\Lambda' }^2 + I) {U}^{-1} \\
& = &\frac{1}{4s}I + \frac{p}{2s} \A\D^{-1} + \frac{1-p}{2s}(\A\D^{-1})^2. \\
\end{eqnarray*}
A similar computation gives us the corresponding $\Sigma$ for the FTSR case.

\

\noindent For $\rho=\frac{1}{2}$, note that in the case of the FTNR model, since the smallest eigenvalue of the matrix $I+p\A\D^{-1} +(1-p)(\A\D^{-1})^2$ is equal to  $1/2$, we have that the eigenvalues of the matrix $I+2p\A\D^{-1} +2(1-p)(\A\D^{-1})^2$ are all non-negative. Hence,
\begin{eqnarray*} \widetilde{\Sigma} &= &\lim_{t \to \infty}  \frac{1}{\log t}\int_{0}^{\log t} {(e^{-(-  \frac{\partial h}{\partial z}(\frac{1}{2} \1)- \frac{I}{2})u})}^{T}    \Gamma e^{-(-  \frac{\partial h}{\partial z}(\frac{1}{2} \1)- \frac{I}{2})u }\,du \  \\
&=& \lim_{t \to \infty}  \frac{1}{4s \log t} U \int_{0}^{\log t} {(e^{-(2p \Lambda'  + 2(1-p){\Lambda' }^2 + I )u})}\,du  {U}^{-1}.\\
\end{eqnarray*}
Since, for $y>0$, $\lim_{t \to \infty}  \frac{1}{\log t}\int_{0}^{\log t} {e^{-yu}}\,du = \lim_{t \to \infty}\frac{1-t^{-y}}{y \log t}=0$ and since $\lim_{t \to \infty}  \frac{1}{\log t}\int_{0}^{\log t} {e^{-yu}}\,du =1$ for $y=0$, the second part of the statement of the theorem clearly holds for the FTNR case. For the FTSR case, the eigenvalues of $(2p+1)I+2(1-p)\A\D^{-1}$ are non-negative and the rest of the computation is similar.
\end{proof}
A similar result can also be obtained for the FTSNR model (albeit with a few extra assumptions than those used below), but we desist from explicitly stating it here. 


\section{Friedman-Type Urns on Directed Graphs} \label{Sec:Dir}
The sampling and reinforcement dynamics described in section~\ref{Sec:prelim} has a natural  analogue on directed graphs. In this case, for each urn, balls are sampled balls from the in-neighbourhood of that urn and the reinforcement is done (to self, out-neighbours or both) depending on the configuration of the drawn sample. We define the in-neighbours and out-neighbours of the $i^{th}$ vertex as $V_i^{in} \coloneqq \{ v : v \to i \}$ and $V_i^{out} \coloneqq \{ v : i \to v \}$, where $u \to w$ denotes the existence of a directed edge from $u$ to $w$. Further, we define the in-degree of $i$ as $d_i^{in} = \vert V_i^{in} \vert$. The $(i, j)^{th}$ entry of the adjacency matrix is given by $\mathbb{I}_{\{i \to j\}}$. We assume that all the vertices have a non-zero in-degree and that there are no self-loops. 

We extend the convergence result in Theorem~\ref{Friedman} to the case of weakly connected directed graphs with no self-loops. Assume that the directed graph $\mathcal{G}$ consists of $r$ strongly connected components. Without loss of generality (after a renumbering of vertices, if required), assume that these components $\mathcal{G}_1,\mathcal{G}_2 \ldots \mathcal{G}_r$ are such that for any directed edge $(l,m)$ between the components, if $l \in \mathcal{G}_i$ and $m \in \mathcal{G}_{j}$ for $ 1 \leq i,j \leq r$, then $i<j$. Since the graph is weakly connected, for every component $\mathcal{G}_j$ with $j > 1$, there exists a $k < j$ such that there is at least one directed edge from a vertex in $G_k$ to a vertex in $\mathcal{G}_j$. A similar recursion for $Z_t$ can be written with $\D$ replaced by $\D_{in}$, the diagonal matrix consisting of in-degrees of each vertex/urn on the diagonal. We restrict our discussion to Friedman-type reinforcement. Similar to \eqref{h functions}, for the corresponding stochastic approximation scheme on directed graphs, we have

\begin{equation} 
h(Z_t)  =  \begin{cases} 
 \1 -Z_t((1+p)I + (1-p)\A \D_{in}^{-1}) & \text{for the FTSR model,}  \\
  \1 -Z_t(I + p\A \D_{in}^{-1}+ (1-p)(\A \D_{in}^{-1})^2 ) & \text{for the FTNR model,} \\
 \1- Z_t\left[(pI + (1-p)\A \D_{in}^{-1})(\A+I)(I+\D_{in})^{-1} +I \right]  & \text{for the FTSNR model.}  
\end{cases}
 \end{equation}
To obtain conditions for convergence, we classify cases for which the matrix $-\nabla h(Z_t)$ 
is invertible and the corresponding ODE $\dot{z} = h(z)$ has a unique stable limit point. For $ 1 \leq i \leq r$, define $\A_i$ and $\D_{in}^{(i)}$ as the adjacency matrix and the in-degree matrix restricted to $\mathcal{G}_i$. Given the above decomposition of the graph, $-\nabla h(Z_t)$ becomes an upper triangular block diagonal matrix of the form 
$$\begin{pmatrix}
    h_1(\A_1,\D_{in}^{(1)}) & * & * & \hdots & \hdots & * \\
     0 & h_2(\A_2,\D_{in}^{(2)}) & *   & * &  \hdots & * \\
      0 & 0  & h_3(\A_3,\D_{in}^{(3)}) & * & \vdots & \vdots & \\
     \vdots & \vdots & 0 & \ddots & \hdots & * \\
     0& 0 & \vdots & 0 & h_{r-1}(\A_{r-1},\D_{in}^{(r-1)}) & * &  \\
     0 &  0 & \hdots  & 0 & 0 & h_r(\A_r,\D_{in}^{(r)})
  \end{pmatrix} $$

\noindent where, for $1 \leq i \leq r$, $h_i(\A_i,\D_{in}^{(i)})$ is a matrix that is a function of $A_i$ and $\D_{in}^{(i)}$ and depends on the type of reinforcement. For instance, for the FTSR model, $h_1(\A_1,\D_{in}^{(1)}) =(1+p)I + (1-p)\A_1 ({\D_{in}^{(1)}})^{-1}$. Note that for $2 \leq i \leq r$,  $h_i(\A_i,\D_{in}^{(i)}) - I$ is a strictly (column) sub-stochastic, irreducible matrix  (because there is at least one incoming edge to the strongly connected component $\mathcal{G}_i$); hence, $h_i(\A_i,\D_{in}^{(i)})$ is invertible for $2 \leq i \leq r$. Thus, the invertibility of  $h_1(\A_1,\D_{in}^{(1)})$ is sufficient for the invertibility of $-\nabla h(Z_t)$. We now state and prove a result analogous to Theorem \ref{Friedman}.

\

\begin{theorem}  \label{Friedman_second} For the Friedman-type reinforcement scheme and for a weakly connected directed graph (with a positive in-degree for each vertex) with $r$ strongly connected components $\mathcal{G}_1,\mathcal{G}_2 \ldots \mathcal{G}_r$ which are arranged as above (so that $\mathcal{G}_1$ has no incoming edges from any other component), as $t \rightarrow \infty$,  $Z_t \rightarrow \frac{1}{2} \1$ almost surely whenever one of the following holds.
\begin{enumerate}
\item FTSR model
 \begin{enumerate}
    \item   $0<p \leq 1$. 
     \item $p=0$ and $\mathcal{G}_1$ is a directed cycle of odd length.
   
     \end{enumerate}
\item FTNR model
\begin{enumerate}
    \item   $0 < p<1$. 
     \item $p=1$ and $\mathcal{G}_1$ is a directed cycle of odd length.
   
     \end{enumerate}
\item The model is FTSNR, regardless of the value of $p$.

\end{enumerate}
\end{theorem}

\begin{proof}
The proofs for 1.(a) and 2.(a) are identically like before since we only require stochasticity of $\A{\D_{in}}^{-1}$ to prove these. For 1.(b) and 2.(b), we need to show that the matrix $\D_{in} +\A$ is invertible for a directed cycle of odd length. For this case, for a column vector $v=(v_1,v_2 \ldots v_N)^{T}$ satisfying $(\D_{in} +\A)v=\0$, we have $v_N+v_1=0$ and $v_i +v_{i+1}=0$ for $1 \leq i  \leq N-1$. When $N$ is odd, this set of equations is satisfied only if $v=0$. For the third part, the proof for $0 < p \leq 1$ also follows like before. For $p=0$, we make use of the structure of the matrix $-\nabla h(Z_t)$ as shown above. We first show that $\A\D_{in}^{-1}\A +\A\D_{in}^{-1}$ is aperiodic for a strongly connected graph. 

Let the length of the directed cycle starting and ending at the vertex labelled $1$ be denoted by $\gamma$ ($\gamma \geq 2$ since self-loops are not allowed). If $\gamma =2$, the matrix above is aperiodic since ($\A\D_{in}^{-1}\A +\A\D_{in}^{-1})_{1,1} \geq (\A\D_{in}^{-1}\A)_{1,1}>0$. For $\gamma \geq 3$, first note that ($\A\D_{in}^{-1}\A +\A\D_{in}^{-1})^{\gamma}_{1,1} \geq (\A\D_{in}^{-1})^{\gamma}_{1,1} >0$. Let the cycle in question be denoted by $1 \rightarrow \ldots x \rightarrow y \rightarrow 1$. Now, $(\A\D_{in}^{-1})^{\gamma-2}_{1,x} >0$ and $(\A\D_{in}^{-1}\A)_{x,1}>0$. Hence, ($\A\D_{in}^{-1}\A +\A\D_{in}^{-1})^{\gamma-1}_{1,1} \geq (\A\D_{in}^{-1})^{\gamma-2}_{1,x}(\A\D_{in}^{-1}\A)_{x,1}>0$, and we are done with this case also. We use this fact to conclude the invertibility of $h_1(\A_1,\D_{in}^{(1)})=(\A_1 (\D_{in}^{(1)})^{-1})(\A_1+I)(I+\D_{in}^{(1)})^{-1} +I$ for the FTSNR case, which further shows that $-\nabla h(Z_t)$ is also invertible for this case.

\end{proof}

Note that for the case of FTSR, when $p=0$ (or FTNR with $p=1$) the class of graphs for which the matrix $\D_{in} +\A$ is full rank is, in fact, much larger. For example, $\D_{in} +\A$ is invertible for the directed graph shown in Figure \ref{fig:odd}. So, for a graph $\mathcal{G}$ with $\mathcal{G}_1$ as in Figure \ref{fig:odd}, $Z_t \rightarrow \frac{1}{2} \1$ almost surely as $t \rightarrow \infty$.

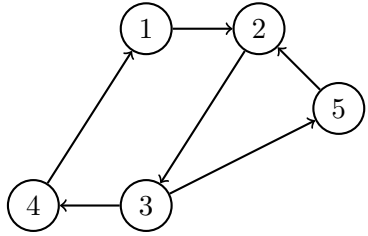
\begin{figure}[h]
\begin{center}
\begin{tikzpicture}[node distance={15mm}, thick, main/.style = {draw, circle}] 
\node[main] (1) {$1$}; 
\node[main] (2) [ right of=1] {$2$}; 
\node[main] (3) [below=2] {$3$}; 
\node[main] (4) [left of=3] {$4$}; 
\node[main] (5) [below right of=2] {$5$};

\draw[->] (1) -- (2);
\draw[->] (2) -- (3); 
\draw[->] (3) -- (4); 
\draw[->] (4) -- (1); 
\draw[->] (3)--(5);
\draw[->]  (5)--(2);

\end{tikzpicture} 
\end{center}

\caption{A strongly connected directed graph on 5 vertices containing an odd cycle $2 \to 3 \to 5 \to 2$. It can be easily verified that the corresponding matrix $\D_{in} +\A$ is invertible for this graph.}
\label{fig:odd}
\end{figure}

We leave a complete classification of the nature of limit points for the case of directed graphs for future work. In general, the limiting behaviour can be very complicated and it is very different from urns on undirected graphs. We illustrate this in the example below. 

We look at an example where $h_1(\A_1,\D_{in}^{(1)})$ (and therefore $-\nabla h(Z_t)$) is not invertible. \color{black} In order to prove a convergence result for this example, we make use of the method developed in Section~\ref{Sec:FTSR}. 

\

\begin{example} \label{ex}
Consider the graph in Figure \ref{fig:gen} with reinforcement type FTSR and $p=0$. Assume that $T_0^{i} =T_0 \ \forall  1 \leq i \leq 5$. Then, $T_t^{i} =T_t^1 \eqqcolon T_t \ \forall\  1 \leq i \leq 5,  t\geq 0$.  

\begin{figure}[h]
\begin{center}
\begin{tikzpicture}[node distance={15mm}, thick, main/.style = {draw, circle}] 
\node[main] (1) {$1$}; 
\node[main] (2) [ right of=1] {$2$}; 
\node[main] (3) [below left of=1] {$3$}; 
\node[main] (4) [ right of=3] {$4$}; 
\node[main] (5) [below right of=3] {$5$}; 

\draw[->] (1.north) to [out=25,in=155] (2.north);
\draw[->] (2.south) to [out=205,in=335] (1.south); 
\draw[->] (1) -- (3);
\draw[->] (3) -- (4); 
\draw[->] (4) -- (5); 
\draw[->] (5) -- (3); 

\end{tikzpicture} 
\end{center}
\caption{A directed graph on 5 vertices. Example~\ref{ex} discusses the FTSR model with $p=0$ on this graph. Observe that if the direction of the edge between nodes 3 and 1 is reversed, the resulting graph falls into the category of graphs that are covered under Theorem~\ref{Friedman_second}.}
\label{fig:gen}
\end{figure}
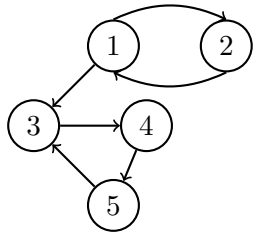
%
%

For the graph in Figure \ref{fig:gen}, we can use neither of the stochastic approximation theorems used earlier to conclude convergence, because neither is the limit point unique nor can we write $h$ function as the gradient of some function since the graph is not regular and the adjacency matrix is also not symmetric. However, on solving for $h(z)=\0$ for this example, we get $z=(3a-1,2-3a,1-a,a,1-a)$ for $a \in \left[\frac{1}{3},\frac{2}{3}\right]$. Motivated by this, we will use arguments similar to those in Section \ref{Sec:FTSR} to conclude that $Z_t$  converges almost surely to a random vector $Z=(Z_1,Z_2\ldots Z_5)$, where $Z_1=3Z_\infty-1$, $Z_2=2-3Z_\infty$, $Z_3=1-Z_\infty$, $Z_4=Z_\infty$ and $Z_5=1-Z_\infty$, where $Z_\infty$ is a random variable supported on  $\left[\frac{1}{3},\frac{2}{3}\right]$.

It suffices to show that $Z_t$ admits an almost sure limit and that $Z_t(1)+Z_t(2) \xrightarrow{a.s.} 1, Z_t(3)+Z_t(4) \xrightarrow{a.s.} 1,Z_t(4)+Z_t(5) \xrightarrow{a.s.} 1$, and $Z_t(1)+3Z_t(3) \xrightarrow{a.s.} 2$, which is equivalent to showing  $Z_t M \xrightarrow{a.s.} (1,1,1,1,1)$ where

\begin{equation*}M=
\begin{pmatrix}
  
  1&0&0&\frac{1}{2}&0 \\
  
  1&0&0&0&0   \\
  0&1&0&\frac{3}{2}&0 \\
  0&1&1&0&1 \\
  0&0&1&0&1 \\
   
  \end{pmatrix}.
\end{equation*}

Consider the stochastic matrix  $\A\D_{in}^{-1}$. Note that $\vv =(1,1,0,0,0)^\top$ is an eigenvector of $\A\D_{in}^{-1}$ corresponding to the eigenvalue $1$. Let $\tilde{Z_t} =Z_t \vv =Z_t\A\D_{in}^{-1} \vv$. Then, recalling \ref{FTSRrecur}, we have
\begin{align}\label{dilda_Z_T}
\tilde{Z}_{t+1} &= Z_{t+1} \vv \nonumber \\
&= \left( \left(1-\frac{Cs}{T_{t+1}}\right) Z_t + \frac{C}{T_{t+1}}\chi_{t+1}\right) \vv \nonumber \\
&= \tilde{Z_{t}} +\frac{C}{T_{t+1}}\left(\chi_{t+1} \vv-s\tilde{Z_{t}}\right)\nonumber \\
&= \tilde{Z_{t}} +\frac{C}{T_{t+1}}\left((\chi_{t+1}-E[\chi_{t+1} | \FF_t]) \vv + E[\chi_{t+1} | \FF_t] \vv -s\tilde{Z_{t}}\right)\nonumber \\
&= \tilde{Z_{t}} +\frac{C}{T_{t+1}}\left((\chi_{t+1}-E[\chi_{t+1} | \FF_t]) \vv + s(\1-Z_t\A\D_{in}^{-1}) \vv-s\tilde{Z_{t}}\right)\nonumber \\
&= \tilde{Z_{t}} +\frac{C}{T_{t+1}}\left(\Delta M_t + s(\1\vv- 2\tilde{Z_{t}})\right),
\end{align}

where $\Delta M_t= (\chi_{t+1}-E[\chi_{t+1} | \FF_t])\vv$ is the martingale difference, and we have used $E[\chi_{t+1} | \FF_t]=s(\1-Z_t\A\D_{in}^{-1})$. Using stochastic approximation theory, we get $Z_t \vv \xrightarrow{a.s.} 1$ (that is, $Z_t(1) +Z_t(2) \xrightarrow{a.s.} 1$). Hence, showing $Z_t M \xrightarrow{a.s.} \1$ is equivalent to showing $Z_tQ \xrightarrow{a.s.} 0$, where $Q=M-\vv \1$. It can be verified that for this example, $0$ is a simple eigenvalue of the matrix $I+\A\D_{in}^{-1}$, and that this matrix is diagonalizable, so we can compute matrices $P$ and $\Lambda$ such that  $I+\A\D_{in}^{-1} = P \Lambda P^{-1}$. Like before, we have numbered vertices of the graph in such a way that the first diagonal element of $\Lambda$ is $0$. Recall from the proof of  Lemma~\ref{M_N} that in the case of  FTSR, when $p=0$, for a matrix $M$, if we have that $e_1^{T}e_1P^{-1}M=0$, we can conclude that $Var(Z_{t}M) \sim \frac{1}{t^{\epsilon}} $ for some $\epsilon > 0$. \color{black} If we also have that $\1M=0$, then we can additionally conclude that $E[Z_t M] \sim \frac{1}{t^{\epsilon^{\prime}}} $ for some $\epsilon^{\prime} > 0$. Now, for this example we have 
\begin{equation*}Q=
\begin{pmatrix}
  
  0&-1&-1&\frac{-1}{2}&-1 \\
  
  0&-1&-1&-1&-1  \\
  0&1&0&\frac{3}{2}&0 \\
  0&1&1&0&1 \\
  0&0&1&0&1 \\
   \end{pmatrix}, \quad  I+\A\D_{in}^{-1}=
  \begin{pmatrix}
  1&1&\frac{1}{2}&0&0 \\
  1&1&0&0&0   \\
  0&0&1&1&0 \\
  0&0&0&1&1 \\
  0&0&\frac{1}{2}&0&1 \\
   \end{pmatrix},
\end{equation*} 

\noindent and 
\begin{equation*}
 e_1^{T}e_1P^{-1}Q=
  \begin{pmatrix}
  1&0&0&0&0 \\
  0&0&0&0&0   \\
  0&0&0&0&0 \\
  0&0&0&0&0 \\
  0&0&0&0&0 \\
   
  \end{pmatrix}
  \begin{pmatrix}
  
  -\frac{1}{2}&\frac{1}{2}&\frac{1}{6}&-\frac{1}{6}&\frac{1}{6} \\
  * & * & * & * & * \\
 \vdots  & \vdots & \vdots & \vdots & \vdots  \\
  * & * & * & * & * \\
  \end{pmatrix}
  \begin{pmatrix}
  0&-1&-1&\frac{-1}{2}&-1 \\
  
  0&-1&-1&-1&-1  \\
  0&1&0&\frac{3}{2}&0 \\
  0&1&1&0&1 \\
  0&0&1&0&1 \\
   \end{pmatrix}
   =\0_{N \times N}.
 \end{equation*}

Clearly $\1Q=\0$. So, $Z_tQ \xrightarrow{\mathcal{L}^2} \0$. Also, note that using \eqref{dilda_Z_T} we have

\begin{eqnarray*}
E[\tilde{Z}_{t+1}-1] &=& E[\tilde{Z}_t] - 1 + \frac{2Cs}{T_{t+1}}  (1 - E[\tilde{Z}_t]) \\\nonumber
&=& \prod_{k=0}^t \left( 1 - \frac{2Cs}{T_{k+1}} \right)(\tilde{Z}_0-1) \\ \nonumber
& \sim & \frac{1}{t^2} (\tilde{Z}_0-1).
\end{eqnarray*}

Hence,

\begin{equation}\label{rate_of_dilda_expectation}
\left|E[\tilde{Z}_{t}-1]\right| \sim  \frac{1}{t^2} \left(\left|\tilde{Z}_0-1\right|\right).
\end{equation}

We now show that $Z_t$ attains an almost sure limit. Indeed, we have

\begin{eqnarray*} \sum_{t=0}^{\infty} E[\|E[Z_{t+1} | \FF_t]-Z_t\|] & = & Cs \sum_{t=0}^{\infty} \frac{ E[\|\1-Z_t(I+\A \D^{-1})\|]}{T_{t+1}}\\
&=& Cs \sum_{t=0}^{\infty} \frac{ E[\|\1-\frac{\tilde{Z}_{t}}{2}\1(I+\A \D^{-1})-(Z_t-\frac{\tilde{Z}_{t}}{2}\1)(I+\A \D^{-1})\|]}{T_{t+1}}\\
&\leq & Cs \sum_{t=0}^{\infty} \left\{ \frac{ E[\|\1-\tilde{Z}_{t}\1\|+  E[\| (Z_t- \frac{\tilde{Z}_{t}}{2}\1)(I+\A \D^{-1})\|]}{T_{t+1}} \right\} \\
&\leq & Cs  \left\{\|1\|\left[ \sum_{t=0}^{\infty} \frac{ \sqrt{Var(\tilde{Z}_{t})}}{T_{t+1}}  +\sum_{t=0}^{\infty} \frac{ |E[\tilde{Z}_{t}-1]|}{T_{t+1}} \right]  \right. \\
&& +   \sum_{t=0}^{\infty} \frac{ \sqrt{\Tr(Var((\frac{\tilde{Z}_{t}}{2}\1-Z_t)(I+\A \D^{-1})))}}{T_{t+1}}\\
&&+ \left.\sum_{t=0}^{\infty} \frac{ \| E[(\frac{\tilde{Z}_{t}}{2}\1-Z_t)(I+\A \D^{-1})]\|}{T_{t+1}} \right\},
\end{eqnarray*}

where we have skipped some of the intermediate steps because they follow like before. Now, note that $(\frac{\tilde{Z}_{t}}{2}\1-Z_t)(I+\A \D^{-1}) =Z_tQ^{\prime}$, where $Q^{\prime}= (\frac{v}{2}\1-I)(I+\A \D^{-1})$. We have

\begin{equation*}
Q^{\prime}=
  \begin{pmatrix}
  0&0&-\frac{1}{2}&-1&-1 \\
  0&0&-1&-1&-1  \\
  0&0&1&1&0 \\
  0&0&0&1&1 \\
  0&0&\frac{1}{2}&0&1 \\
  \end{pmatrix}
 \end{equation*}
 
and  
\begin{equation*}
 e_1^{T}e_1P^{-1}Q^{\prime}=
  \begin{pmatrix}
  1&0&0&0&0 \\
  0&0&0&0&0   \\
  0&0&0&0&0 \\
  0&0&0&0&0 \\
  0&0&0&0&0 \\
   \end{pmatrix}
  \begin{pmatrix}
  
  -\frac{1}{2}&\frac{1}{2}&\frac{1}{6}&-\frac{1}{6}&\frac{1}{6} \\
  * & * & * & * & * \\
 \vdots  & \vdots & \vdots & \vdots & \vdots  \\
  * & * & * & * & * \\
  \end{pmatrix}
  \begin{pmatrix}
  0&0&-\frac{1}{2}&-1&-1 \\
  0&0&-1&-1&-1  \\
  0&0&1&1&0 \\
  0&0&0&1&1 \\
  0&0&\frac{1}{2}&0&1 \\
  \end{pmatrix}
   =\0_{N \times N}.
 \end{equation*}

Clearly $\1 Q^{\prime} =0$. So, 

 $$\sum_{t=0}^{\infty} \frac{ \| E[(\frac{\tilde{Z}_{t}}{2}\1-Z_t)(I+\A \D^{-1})]\|}{T_{t+1}} < \infty \text{ and } \sum_{t=0}^{\infty} \frac{ \sqrt{\Tr(Var((\frac{\tilde{Z}_{t}}{2}\1-Z_t)(I+\A \D^{-1})))}}{T_{t+1}} < \infty.  $$
 
Also, the matrix $v\1$ satisfies $e_1^{T}e_1P^{-1}v\1 = \0_{N \times N}$. So,  $\sum_{t=0}^{\infty} \frac{ \sqrt{Var(\tilde{Z}_{t})}}{T_{t+1}} < \infty$. Combining all this with  \eqref{rate_of_dilda_expectation}, we have $\sum_{t=0}^{\infty} \frac{ |E[\tilde{Z}_{t}-1]|}{T_{t+1}}  < \infty$. So, $Z_t$ admits an almost sure limit. The rest of the arguments are identical to those in the proof of Theorem~\ref{syncFTSR}. Hence, we have that  $Z_tQ \xrightarrow{a.s.} \0$. 


\end{example}


\begin{appendix}

\section{Stochastic Approximation Theory}\label{appn}



Stochastic approximation has been used extensively in recent times to analyse the behaviour of urn processes (for instance, see \cite{MR3780393,laruelle2017addendum}). A recursion of the form
  \begin{equation}\label{basicSA}
 x_{t+1} = x_t + a_{t+1} \left( h(x_t) +  \Delta M_{t+1} +\epsilon_{t+1} \right), \qquad \forall t \geq 0
\end{equation}
  is called a stochastic approximation scheme if it satisfies the following conditions.
\begin{enumerate}
\item $\sum_{t\geq 1} a_t = \infty$ and $\sum_{t\geq 1} a_t^2 < \infty$.
\item $h: \R^d \to \R^d$ is a Lipschitz function. 
\item There exists a constant $\delta >0$, such that  $E\left [\|\Delta M_{t+1}\|^2 \vert \FF_t\right] \leq \delta (1+\|x_t\|^2)$, almost surely  $\forall t \geq 0$,  where $\FF_t = {\sigma (x_0, M_1, \ldots, M_t)}$.
\item $\sup_{t\geq 0} \|x_t\| < \infty$, almost surely.
\item $\{ \epsilon_t \}_{t\geq 1}$ is a  bounded sequence such that $\epsilon_t \to 0$, almost surely as $t\to \infty$.
\end{enumerate}	 
The main result (see Chapter 2 of \cite{borkar_stochastic_2008}) implies that the iterates of the recursion for $x_t \in \mathbb{R}^d$ satisfying  conditions (1)-(5)  converge almost surely to the stable limit points of the solutions of an ordinary differential equation given by $\dot{x}_t = h(x_t)$. Recall that $x^*$ is a stable limit point of the O.D.E. $\dot{x}_t = h(x_t)$ if the real parts of all the eigenvalues of the Jacobian $\frac{\partial h}{\partial z}(x^*)$ are non-positive. 

\

\noindent Further, suppose there exists a differentiable function $f: \mathbb{R}^d \to \mathbb{R}$ such that $h = -\nabla f$. The stochastic approximation scheme can then be written as 
\begin{equation} \label{SA2}
x_{t+1} = x_t - a_{t+1} ( \nabla f(x_t) + \beta_t),
\end{equation}
where $\{\beta_t \}_{t \geq 0}$ is an $\mathbb{R}^{d}$- valued stochastic process. This recursion is also known as the stochastic gradient descent. We now state the main convergence theorem from \cite{MR3315611}.  Let $S$ be the set of stationary points of $f(\cdot)$, i.e. $ S = \{ x \in \mathbb{R}^{d} : \nabla f(x) = 0 \}$. Sequence $\{\gamma_t \}_{t \geq 0}$ is defined by $\gamma _ 0 = 0$ and $\gamma_t = \sum_{i=0} ^{t-1} a_i$. 
For $ t \in (0, \infty)$ and $n \geq 0$, $\tilde {a}(n, t)$ is the integer defined as $\tilde{a}(n,t) = \max\{ k \geq n : \gamma_k - \gamma_n \leq t \}$. Then,

\begin{theorem} (Theorem 2.1 \cite{MR3315611}) \label{Tadic_SA} Suppose the following conditions hold. 
\begin{enumerate}
\item $\lim_{t \to \infty} a_t = 0$ and  $\sum_{t = 0} ^{\infty} a_t = \infty$.

\item There exists a real number $r \in (1, \infty)$ such that on $\{ \sup_{t \geq 0} \left|\left| x_t \right|\right| < \infty \}$,
\begin{equation*} 
P \left(\limsup_{n \to \infty} \max \limits_{ n \leq k < \tilde{a}(n,1)} \Big \| \sum_{i = n} ^{k} a_i {\gamma_i}^r \beta_i \Big \| < \infty \right) = 1. 
\end{equation*}
 \item For any compact set $Q \subset \mathbb{R}^{d}$ and any $c \in f(Q)$, there exist real numbers $\delta_{Q, c} \in (0,1], $
$M_{Q, c} \in [1, \infty)$ such that 
\begin{equation*} 
\left|f(x) - c \right| \leq M_{Q, c} {\left|\left|\nabla f(x) \right|\right|} ^{\mu_{Q, c}} 
\end{equation*}
for all $x \in Q$ satisfying $\left|f(x) - c \right| \leq \delta_{Q, c}.$
\end{enumerate}
Then, $x^* = \lim_{t \to \infty} x_t$ exists and satisfies $\nabla f(x^*) = \0 $ with probability $1$ on $\{ \sup_{t \geq 0}\| x_t  \| < \infty \}$.
\end{theorem}

\end{appendix}

\section*{Acknowledgment}The first author was supported by IISER Mohali doctoral fellowship. The second author was supported in part by DST-INSPIRE fellowship and SERB-MATRICS grant.

\bibliographystyle{unsrt}

\bibliography{References}

\begin{thebibliography}{10}

\bibitem{MR3780393}
Nabil Lasmar, C\'{e}cile Mailler, and Olfa Selmi.
\newblock Multiple drawing multi-colour urns by stochastic approximation.
\newblock {\em J. Appl. Probab.}, 55(1):254--281, 2018.

\bibitem{MR3654806}
Antar Bandyopadhyay and Debleena Thacker.
\newblock P\'{o}lya urn schemes with infinitely many colors.
\newblock {\em Bernoulli}, 23(4B):3243--3267, 2017.

\bibitem{MR3666709}
Markus Kuba and Hosam~M. Mahmoud.
\newblock Two-color balanced affine urn models with multiple drawings.
\newblock {\em Adv. in Appl. Math.}, 90:1--26, 2017.

\bibitem{MR3915423}
Aguech Rafik, Lasmar Nabil, and Selmi Olfa.
\newblock A generalized urn with multiple drawing and random addition.
\newblock {\em Ann. Inst. Statist. Math.}, 71(2):389--408, 2019.

\bibitem{MR2203815}
May-Ru Chen and Ching-Zong Wei.
\newblock A new urn model.
\newblock {\em J. Appl. Probab.}, 42(4):964--976, 2005.

\bibitem{MR3126664}
Markus Kuba, Hosam Mahmoud, and Alois Panholzer.
\newblock Analysis of a generalized {F}riedman's urn with multiple drawings.
\newblock {\em Discrete Appl. Math.}, 161(18):2968--2984, 2013.

\bibitem{MR3217785}
Paolo Dai~Pra, Pierre-Yves Louis, and Ida~G. Minelli.
\newblock Synchronization via interacting reinforcement.
\newblock {\em J. Appl. Probab.}, 51(2):556--568, 2014.

\bibitem{crimaldi2016fluctuation}
Irene Crimaldi, Paolo Dai~Pra, and Ida~Germana Minelli.
\newblock Fluctuation theorems for synchronization of interacting p{\'o}lya’s
  urns.
\newblock {\em Stochastic processes and their applications}, 126(3):930--947,
  2016.

\bibitem{aletti2017interacting}
Giacomo Aletti and Andrea Ghiglietti.
\newblock Interacting generalized {F}riedman’s urn systems.
\newblock {\em Stochastic Processes and their Applications}, 127(8):2650--2678,
  2017.

\bibitem{sahasrabudhe2016synchronization}
Neeraja Sahasrabudhe.
\newblock Synchronization and fluctuation theorems for interacting {F}riedman
  urns.
\newblock {\em Journal of Applied Probability}, 53(4):1221--1239, 2016.

\bibitem{kaur2023interacting}
Gursharn Kaur and Neeraja Sahasrabudhe.
\newblock Interacting urns on a finite directed graph.
\newblock {\em Journal of Applied Probability}, 60(1):166--188, 2023.

\bibitem{MR3269167}
Jun Chen and Cyrille Lucas.
\newblock A generalized {P}\'{o}lya's urn with graph based interactions:
  convergence at linearity.
\newblock {\em Electron. Commun. Probab.}, 19:no. 67, 13, 2014.

\bibitem{MR3346459}
Michel Bena\"{\i}m, Itai Benjamini, Jun Chen, and Yuri Lima.
\newblock A generalized {P}\'{o}lya's urn with graph based interactions.
\newblock {\em Random Structures Algorithms}, 46(4):614--634, 2015.

\bibitem{MR2079914}
Anna~Maria Paganoni and Piercesare Secchi.
\newblock Interacting reinforced-urn systems.
\newblock {\em Adv. in Appl. Probab.}, 36(3):791--804, 2004.

\bibitem{crimaldi2019synchronization}
Irene Crimaldi, Paolo Dai~Pra, Pierre-Yves Louis, and Ida~G Minelli.
\newblock Synchronization and functional central limit theorems for interacting
  reinforced random walks.
\newblock {\em Stochastic processes and their applications}, 129(1):70--101,
  2019.

\bibitem{10.1214/17-AAP1296}
Giacomo Aletti, Irene Crimaldi, and Andrea Ghiglietti.
\newblock {Synchronization of reinforced stochastic processes with a
  network-based interaction}.
\newblock {\em The Annals of Applied Probability}, 27(6):3787 -- 3844, 2017.

\bibitem{10.3150/18-BEJ1092}
Giacomo Aletti, Irene Crimaldi, and Andrea Ghiglietti.
\newblock {Networks of reinforced stochastic processes: Asymptotics for the
  empirical means}.
\newblock {\em Bernoulli}, 25(4B):3339 -- 3378, 2019.

\bibitem{MR1918746}
Xiao-Dong Zhang and Jiong-Sheng Li.
\newblock The {L}aplacian spectrum of a mixed graph.
\newblock {\em Linear Algebra Appl.}, 353:11--20, 2002.

\bibitem{MR3315611}
Vladislav~B. Tadi\'{c}.
\newblock Convergence and convergence rate of stochastic gradient search in the
  case of multiple and non-isolated extrema.
\newblock {\em Stochastic Process. Appl.}, 125(5):1715--1755, 2015.

\bibitem{godsil2001}
Chris Godsil and Gordon Royle.
\newblock {\em The Laplacian of a Graph}, pages 279--306.
\newblock Springer New York, New York, NY, 2001.

\bibitem{Metivier+1982}
Michel Métivier.
\newblock {\em Semimartingales: A Course on Stochastic Processes}.
\newblock De Gruyter, Berlin, New York, 1982.

\bibitem{MR3582813}
Li-Xin Zhang.
\newblock Central limit theorems of a recursive stochastic algorithm with
  applications to adaptive designs.
\newblock {\em Ann. Appl. Probab.}, 26(6):3630--3658, 2016.

\bibitem{laruelle2017addendum}
Sophie Laruelle and Gilles Pag{\`e}s.
\newblock Addendum and corrigendum to ``randomized urn models revisited using
  stochastic approximation".
\newblock {\em The Annals of Applied Probability}, pages 1296--1298, 2017.

\bibitem{borkar_stochastic_2008}
Vivek~S. Borkar.
\newblock {\em Stochastic approximation: a dynamical systems viewpoint}.
\newblock Cambridge University Press ; Hindustan Book Agency, Cambridge, UK :
  New York : New Delhi, 2008.
\newblock OCLC: ocn231580915.

\end{thebibliography}

\end{document}